\documentclass{amsart}
\usepackage{amsmath, amssymb}
\usepackage{bm}
\usepackage{booktabs}
\usepackage{enumitem}
\usepackage{float}
\usepackage{hyperref}
\usepackage[makeroom]{cancel}
\usepackage{multirow}
\usepackage[numbers]{natbib}
\usepackage[title]{appendix}

\setlist{nosep}
\newtheorem{theorem}{Theorem}[section]
\newtheorem{observation}[theorem]{Observation}
\newtheorem{proposition}[theorem]{Proposition}
\newtheorem{lemma}[theorem]{Lemma}
\newtheorem{corollary}[theorem]{Corollary}
\theoremstyle{definition}
\newtheorem{definition}[theorem]{Definition}

\theoremstyle{remark}

\numberwithin{equation}{section}
\begin{document}
\title[TPP Triples of Subgroups and Finite Groups of Class $2$]{On the triple product property for subgroups of finite groups of nilpotency class $2$}
\author{Sandeep R. Murthy}
\address{}
\curraddr{}
\email{srm@tuta.com}
\date{16 January 2026}
\dedicatory{}

\begin{abstract}
{Some non-trivial upper bounds are proved relating to the triple product property (TPP) for subgroups of finite groups of nilpotency class $2$ using elementary group theory. The TPP is the property defined for three non-empty subsets $S, T, U$ of a group $G$ that the group equation $s's^{-1}t't^{-1}u'u^{-1} = 1$, over pairs of elements $s', s \in S$, $t', t \in T$, $u', u \in U$, is satisfied if and only if $s' = s$, $t' = t$, $u' = u$.

When $G$ is finite, and the parameter $\rho_0(G)$, called \emph{subgroup TPP ratio}, is defined as $\rho_0(G) := \max\frac{|S||T||U|}{|G|}$, where the maximum is over the collection of all triples of subgroups $S, T, U$ of $G$ satisfying the TPP, this paper proves that \textup{(1)} $\rho_0(G) < \sqrt{G:Z(G)}$} for (all) groups of nilpotency class $2$, \textup{(2)} $\rho_0(G) \leq p$ for $p$-groups with a cyclic commutator subgroup of order $p$, \textup{(3)} $\rho_0(G) = 1$ for $p$-groups of nilpotency class $2$ with a "large" centre such that either $G/Z(G)$ is isomorphic to $C_p \times C_p$ or $C_{p^2} \times C_p$, \textup{(4)} and $\rho_0(G) = 1$ for $p$-groups of nilpotency class $2$ with "small" (irreducible, complex) character degrees of $1$ or $p$ where $G/Z(G)$ is not isomorphic to $C_p \times C_p \times C_p$.
\end{abstract}

\maketitle

\tableofcontents

\section{Introduction}

This paper addresses the following question(s):

\begin{quote}\label{Q1} If $G$ is \textup{(I)} a (finite) group of nilpotency class $2$, including, in particular, \textup{(II)} a $p$-group with a cyclic commutator subgroup of order $p$, or \textup{(III)} a $p$-group of nilpotency class $2$ with a "large" centre ($p^2 \leq G:Z(G) \leq p^3$), or \textup{(IV)} a $p$-group of nilpotency class $2$ with "small" character degrees ($\text{cd}(G) = \{1, p\}$), how large can the ratio $\rho_0(G) := \max \frac{|S||T||U|}{|G|}$ be, where the maximum is over the collection of all triples of subgroups $S, T, U$ of $G$ satisfying the so-called triple product property (TPP)? 
\end{quote}

These were motivated by patterns in computationally-derived TPP data for a number of "small" groups obtained originally by I. Hedtke and this author in 2012 \citep[Tables 1-4]{HM}, which included groups of various nilpotency classes, as well as non-nilpotent groups. The data suggested the following propositions.

\textup{(1)} Groups with abelian normal subgroups of prime index $p$ satisfy $\rho_0 \leq \frac{p^2}{2p - 1}$, and, in particular, $p$-groups with abelian normal (and maximal) subgroups of index $p$ satisfy $\rho_0 = 1$.

\vspace{1mm}

\textup{(2)} Groups with cyclic normal subgroups of prime index $p$ satisfy $\rho \leq \frac{p^2}{2p - 1}$, where $\rho$ is a generalisation of $\rho_0$ to arbitrary subsets satisfying the TPP.

\vspace{1mm}

\textup{(3)} Groups of nilpotency class $2$ satisfy $\rho_0 < \sqrt{G:Z(G)}$ where $Z(G)$ is the group centre.

\vspace{1mm}

\textup{(4)} $p$-Groups with a cyclic commutator subgroup $G' = [G, G]$ of order $p$, including extraspecial groups, which are necessarily of nilpotency class $2$, satisfy $\rho_0 \leq p$.

\vspace{1mm}

\textup{(5)} $p$-Groups of nilpotency class $2$ with a large centre such that $p^2 \leq G/Z(G) \leq p^3$ satisfy $\rho_0 = 1$.

\vspace{1mm}

\textup{(6)} $p$-Groups of nilpotency class $2$ with small character degrees of $1$ or $p$ satisfy $\rho_0 = 1$.

\vspace{1mm}

Proposition \textup{(1)} was proved in \citep[Theorem 4.1, Corollary 4.3]{Mur2025}, while \textup{(2)} was conjectured for $p = 2$ in \citep[Conjecture 7.6]{HM} and for arbitrary $p$ in \citep[Conjecture 5.1]{Mur2025}. This paper proves propositions \textup{(3)-(4)} and restricted versions of propositions \textup{(5)-(6)} in sections \ref{SecGroupsOfTypeI}, \ref{SecGroupsOfTypeII}, \ref{SecGroupsOfTypeIII}, \ref{SecGroupsOfTypeIV}, with sections referencing data tables contained in the \hyperref[appendix]{Appendix} summarising the relevant computational data, where appropriate. There is a preliminary section containing definitions, basic facts and technical results needed for the main results, some of which include general results that characterise TPP triples of arbitrary finite groups, and not limited to groups of nilpotency class $2$.

\vspace{1mm}

Definitions and characterisations of nilpotency and nilpotency classes of groups, for example, the most standard one in terms of central series, are taken as given. The focus of this paper is groups of nilpotency class $2$, which can be characterised as nonabelian groups satisfying the condition

\begin{equation}\label{NilpClassTwo}
\{1\} < G' \leq Z(G) < G.
\end{equation}

Particularly relevant for the approach used in this paper is that in such groups the centre $Z(G)$ is non-trivial, commutators are central, the quotient $G/Z(G)$ is abelian, and there is an abelian quotient by any (normal) subgroup containing the centre. Bounds for the subgroup TPP ratio can then be derived by studying how the members of large subgroup TPP triples intersect with, or do not intersect with, the centre, and other key normal subgroups, such as the commutator subgroup $G'$.

\vspace{1mm}

Applications of the TPP for the complexity of matrix multiplication, for which it was originally introduced \citep{CU}, are not directly considered in the paper, but the upper bounds derived for the subgroup TPP ratio for the various classes of groups that are presented here are very relevant for such applications, for example, by leading to bounds for the parameter introduced in \cite[Definition 3.1, Lemma 3.1]{CU} for groups called the \emph{pseudoexponent} $\alpha$ that measures how effective a given group "realises" matrix multiplication. This paper continues the combinatorial and group-theoretic approach described in \citep{Neu2011}, and pursued in \citep{Mur2025}.

\section{Preliminaries}\label{Preliminaries}

Standard notation is used. Groups are generally finite, unless stated otherwise, and are typically denoted by $G$, and subgroups by $H, K, M, N$ etc. or, in the case of subsets, including subgroups, satisfying the TPP, by $S, T, U$. The symbol $\cong$ usally denotes group isomorphism, but occasionally denotes set bijection as indicated. The (multiplicative) cyclic group of order $n$ is noted by $C_n$, the ring of integers modulo $n$ by $\mathbb{Z}_n$, and the finite (Galois) field with $p^n$ elements for some prime $p$ by $GF(p^n)$. The term "$p$-group" denotes a finite group of order $p^n$ for some prime $p$ and positive integer $n$. All references to characters are to complex irreducible characters of finite groups. The set of unique character degrees of a group $G$ is denoted by $\text{cd}(G) = \{\chi(1) \mathrel | \chi \in \text{Irr}(G)\}$, where $\text{Irr}(G)$ denotes the set of unique characters of $G$.

A number of basic definitions, notions and technical results used in the sections of main results will be recalled or proved, starting with the TPP \citep[Definition 2.1]{CU}.

\begin{definition}\label{TPP} Let $G$ be a group, finite or infinite, and $S, T, U$ non-empty subsets of $G$ with cardinalities $|S|, |T|, |U|$ respectively.  The subsets $S,T,U$ are said to satisfy the \emph{triple product property} (TPP) if there is an equivalence

\begin{equation}\label{TPPEqn} s's^{-1}t't^{-1}u'u^{-1} = 1 \iff s = s', t = t', u = u' \end{equation}

for all $s, s' \in S$, $t,t' \in T$, $u,u' \in U$, in which case the triple $(S, T, U)$ is said to be a TPP triple of $G$, and the cardinality triple $(|S|,|T|,|U|)$ is said to be the \emph{parameter type}, or simply, the \emph{type}, of the triple $(S, T, U)$, and the product $|S||T||U|$ is said to be the \emph{size} of the triple.  If $S, T, U$ are subgroups of $G$ then $(S,T,U)$ is called a \emph{subgroup TPP triple} of $G$, for which the TPP condition simplifies to

\begin{equation} stu = 1 \iff s = t = u = 1, \end{equation}

for all $s \in S$, $t \in T$, $u \in U$. When $G$ is finite the TPP triple $(S, T, U)$ is be said to be of \emph{trivial} size or \emph{non-trivial} size depending on whether $|S||T||U| \leq |G|$ or $|S||T||U| > |G|$.
\end{definition}

Note that the term "subgroup TPP triple" refers only to TPP triples of subgroups, while "TPP triple" refers to triples composed of arbitrary (non-empty) subsets satisfying the TPP.

\vspace{1mm}

The TPP has two important invariance properties, as noted in the result below (reproduced from \citep[Observation 2.3]{Mur2025}).

\begin{observation} \label{Invariance} Let $(S,T,U)$ be a TPP triple of a group $G$.
\begin{enumerate}
\vspace{1mm}
\item If $\pi \in S_3$, that is, $\pi$ is a permutation of $\{S, T, U\}$, then $(S^\pi,T^\pi,U^\pi)$ is a TPP triple of $G$ (permutation invariance) \textup{\citep[Lemma 2.1]{CU}}.
\vspace{1mm}
\item If $a,b,c,d \in G$ are any elements then $(dSa,dTb,dUc)$ is a TPP triple of $G$ (translation invariance), where $dSa$, $dTb$, $dUc$ are just set products $dSa = \{dsa \mid s \in S\}$, $dTb = \{dtb \mid t \in T\}$, $dUc = \{duc \mid u \in U\}$ respectively \textup{\citep[Observation 2.1]{Neu2011}}.
\end{enumerate}
\end{observation}

Translation invariance for a TPP triple $(S, T, U)$ of a group $G$ means that the members of the triple can always be translated appropriately - specifically, via translations of the form $(S, T, U) \longmapsto (Ss^{-1}, Tt^{-1}, Uu^{-1})$ for non-identity elements $s \in S$, $t \in T$, $u \in U$ - to derive another TPP triple $(S', T', U')$ of the same type, that is, $(|S'|, |T'|, |U'|) = (|S|, |T|, |U|)$, with the property that:

\begin{align}1 \in S' \cap T' \cap U'.
\end{align}

Such TPP triples are called \emph{basic} \citep[Observation 2.1]{Neu2011}. Subgroup TPP triples are necessarily basic. All references to TPP triples shall be to basic TPP triples, unless otherwise stated.

There is a fundamental set-theoretic characterisation of the TPP that is important to note.

\begin{theorem} \label{HM} \textup{\citep[Theorem 3.1]{HM}} In a group $G$ three non-empty subsets $S,T,U \subseteq G$ satisfy the TPP if and only if 
\begin{equation}
\label{BasicTPPC2}Q(S^{\pi}) \cap Q(T^{\pi})Q(U^{\pi})  = Q(T^{\pi}) \cap Q(U^{\pi}) = \{1\}.
\end{equation}
where $\pi \in \text{Sym}(3)$ is a permutation of $\{S, T, U\}$ and $Q(X) := \{x'x^{-1} \mathrel | x, x' \in X\}$ is called the quotient set of a non-empty subset $X\subseteq G$.  If $S, T, U$ are subgroups then this condition reduces to
\begin{equation}
\label{BasicTPPSub}S^{\pi} \cap T^{\pi}U^{\pi}  = T^{\pi} \cap U^{\pi} = \{1\}.
\end{equation}
\end{theorem}

The members of a TPP triple $(S, T, U)$ must satisfy certain combinatorial constraints, as expressed by the inequality \citep[Observation 3.1]{Neu2011}:

\begin{equation}
\label{TPPParamConstraints}|S^{\pi}|(|T^{\pi}| + |U^{\pi}| - 1) \leq |G|, \hspace{3em}\pi \in \text{Sym}(3)
\end{equation}

\vspace{1mm}

If $G$ is a group, the \emph{subgroup TPP ratio} of $G$ is defined as

\begin{equation}\label{TPPRho0}
\rho_0(G) := \frac{\beta_0(G)}{|G|}
\end{equation}

where, following \citep{Neu2011, Mur2025}, $\beta_0(G)$ is the \emph{subgroup TPP capacity of $G$} defined as

\begin{equation}\label{TPPBeta0}
\beta_0(G) := \max \;\{|S||T||U| \mid (S,T,U) \textrm{ is a subgroup TPP triple of }G\}.
\end{equation}.

The second, more general quantity, called \emph{TPP ratio}, for $G$, is defined as

\begin{equation}\label{TPPRho}
\rho(G) := \frac{\beta(G)}{|G|}
\end{equation}

where $\beta(G)$ is the \emph{TPP capacity of $G$} defined as

\begin{equation}\label{TPPBeta}
\beta(G) := \max \;\{|S||T||U| \mid (S,T,U) \textrm{ is a TPP triple of }G\}. 
\end{equation}

As $G$ always realises the trivial (subgroup) TPP triple $(G,\{1\},\{1\})$ then $|G| \leq \beta_0(G) \leq \beta(G)$ and $1 \leq \rho_0(G) \leq \rho(G)$. For abelian groups there is a basic but key fact.

\begin{theorem}\textup{\citep[from the proof of Lemma 3.1]{CU}}\label{ThmRhoAbelianGroups} If $G$ is an abelian group then $\rho_0(G) = \rho(G) = 1$, or, equivalently, $\beta_0(G) = \beta(G) = |G|$.
\end{theorem}
\begin{proof} If $(S, T, U)$ is a TPP triple of an abelian group $G$ then the "triple product map" $S \times T \times U \longrightarrow G$ given by $(s, t, u) \longmapsto stu$ is injective.
\end{proof}

A further definition is useful in differentiating TPP triples by combinations of parameter type and size.

\begin{definition}\label{ProperTPP} A TPP triple $(S, T, U)$ of a group $G$ is said to be \emph{proper} if $\min\{|S|, |T|, |U|\} > 1$, or \emph{improper} if $\min\{|S|, |T|, |U|\} = 1$, or \emph{non-trivial} if $|S||T||U| > |G|$, or \emph{trivial} if $|S||T||U| \leq |G|$, or \emph{maximal} if $\rho(G) = \frac{|S||T||U|}{|G|}$ or, in case $(S, T, U)$ is a maximal subgroup TPP triple, if $\rho_0(G) = \frac{|S||T||U|}{|G|}$.
\end{definition}

Note that a maximal TPP triple need not be unique or non-trivial, and a non-trivial TPP triple need not be unique or maximal. But an improper TPP triple is always trivial \citep[Proof of Lemma 3.1]{CU}.

\vspace{1mm}

The notion of permutability of subsets of groups is quite important for the TPP, and is formally defined below.

\begin{definition} Two non-empty subsets $X, Y$ of a group $G$ are said to \emph{permute} (in $G$), or, equivalently, be \emph{permutable} (in $G$), if $XY = YX$, where $XY = \{xy \mathrel | x \in X, y \in Y\}$ and $YX = \{yx \mathrel | x \in X, y \in Y\}$. Clearly, this notion is symmetric, as $X$ permuting with $Y$ implies that $Y$ permutes with $X$.
\end{definition}

One form of permutability is when a subset normalises another: that is, for non-empty subsets $X, Y \subseteq G$ if $Xy = yX$ for all $y \in Y$ then $Y \subseteq N_G(X)$, where $N_G(X)$ is the normaliser of $X$ in $G$, and $XY = \bigcup_{y \in Y}Xy = \bigcup_{y \in Y}yX = YX$. Of course, if $X$ is a normal subgroup of $G$ then $N_G(X) = G$. A stronger form of permutability in groups, which implies normalisation, is centralisation, where, using the sets $X, Y$ above, $xy = yx$ for all $x \in X$ and $y \in Y$, in which case $Y \subseteq C_G(X)$, where $C_G(X)$ is the centraliser of $X$ in $G$ and $C_G(X) \leq N_G(X)$.

\begin{proposition}\label{PropTPPPermSubsets}Let $(S, T, U)$ be a non-trivial TPP triple of a group $G$, that is, one such that $|S||T||U| > |G|$. Then

\vspace{1mm}

\textup{(1)} No member $S, T, U$ permutes with another member and its inverse, that is, for any distinct $X, Y \in \{S, T, U\}$ either $XY \neq YX$ or $XY^{-1} \neq Y^{-1}X$.

\vspace{1mm}

\textup{(2)} No member $S, T, U$ is normalised by (or, in particular, centralised by) by $G$ or by each other, that is,  $N_G(X) < G$ for any $X \in \{S, T, U\}$, and $Y, Z \not \subseteq N_G(X)$ for any distinct $X, Y, Z \in \{S, T, U\}$. In particular, if $(S, T, U)$ is a subgroup TPP triple then $S, T, U$ are all non-normal subgroups.

\vspace{1mm}

\textup{(3)} If $H$ is the (proper) normaliser (in $G$) of any member of $\{S, T, U\}$ and has index $n = |G: H|$ then $|S|$, $|T|$, $|U|$ satisfy
\begin{equation}
\frac{|S||T||U|}{|G|} \leq n\rho(H)
\end{equation}
and, furthermore, if $H$ is abelian, then
\begin{equation}
\frac{|S||T||U|}{|G|} \leq n.
\end{equation}
\end{proposition}
\begin{proof}Let $G$ and the TPP triple $(S, T, U)$ be as given.

\textup{(1)} Let $S$ permute with $T$ and $T^{-1}$. Then $Q(ST) = ST(ST)^{-1} = STT^{-1}S^{-1} = TT^{-1}SS^{-1} = Q(T)Q(S)$, where $Q(T)Q(S) \cap Q(U) = \{1\}$ (by the TPP for the permuted triple $(T, S, U)$), and so by \citep[Lemma 3.1]{CU} $(ST, U, \{1\})$ is a TPP triple of $G$ such that $|S||T||U| \leq |G|$, a contradiction.\footnote{For two non-empty subsets $X, Y \subseteq G$ the quotient set $Q(XY)$ of their ordinary set product $XY$ is given by $Q(XY) = (XY)(XY)^{-1} = \{(x'y')(xy)^{-1} \mid x', x \in X \text{ and } y', y \in Y\} = \{x'y'y^{-1}x^{-1} \mid x', x \in X \text{ and } y', y \in Y\} = XYY^{-1}X = XQ(Y)X^{-1}$, where, if, additionally, $X$ permutes with $Y$ and $Y^{-1}$, then the right-hand side becomes $XQ(Y)X^{-1} = Q(Y)XX^{-1} = Q(Y)Q(X)$.} Thus $S$ does not permute with $T$ and $T^{-1}$, and, analogously, $S$ does not permute with $U$ and $U^{-1}$, and $T$ does not permute with $U$ and $U^{-1}$.

\vspace{1mm}

\textup{(2)} Let $S$ be normalised by $G$, that is, $N_G(S) = G$. Then $S$ permutes with $T$ and $T^{-1}$, and by part \textup{(1)} $|S||T||U| \leq |G|$, a contradiction. So $N_G(S) < G$, that is, $S$ is normalised by a proper subgroup of $G$. Analogously, $T$ and $U$ are not normalised by $G$, and $N_G(T), N_G(U) < G$. This means that if $S, T, U$ are subgroups they are all non-normal in $G$. In particular, $S$, $T$, and $U$ are not centralised by $G$, or by each other.

\vspace{1mm}

\textup{(3)} Let $H = N_G(S)$ and $n = |G:H|$. By part \textup{(1)}, $N_G(S) < G$, and also $T, U \not \subseteq H$, that is, $T$ and $U$ are not contained in $H$ (otherwise $S$ would permute with $T$ and $T^{-1}$, or with $U$ and $U^{-1}$). Let $G/H$ be the collection of (left) cosets of $H$ in $G$, and let $\overline{S}, \overline{T}, \overline{U} \subseteq G/H$ be the $H$-supports of $S, T, U$ respectively, defined as in \citep[Definition 3.1]{Mur2025}, with sizes $\sigma = |\overline{S}|$, $\tau = |\overline{T}|$, $\tau = |\overline{U}|$ respectively. Now, $\sigma = 1$ as $S \subseteq H$, while $1 < \tau, \upsilon \leq n$ as $T, U \not\subseteq H$. Also, let $T_0 = T \cap H$ and $U_0 = U \cap H$. Using \citep[Observation 4.1]{Neu2011} it can be supposed that $|T_0| \geq |T_y = T \cap yH|$ and $|U_0| \geq |U_y = U \cap zH|$ for any $yH \in \overline{T}$ and $zH \in \overline{U}$. The collections $\{T_y\}_{yH \in \overline{T}}$ and $\{U_z\}_{zH \in \overline{U}}$ form partitions of $T$ and $U$ respectively, and there is a bound for $|S||T||U|$ given by
\begin{align*}|S||T||U| &= |S|\sum_{yH \in \overline{T}}|T_y|\sum_{zH \in \overline{U}}|U_z| \\
                        &\leq \tau\upsilon|S||T_0||U_0| \\
                        &\leq n^2\beta(H).
\end{align*}

Dividing both sides by $|G| = n|H|$ yields
\begin{align*}\frac{|S||T||U|}{|G|} &\leq \frac{n^2\beta(H)}{|G|} \\
                                    &= \frac{n^2\beta(H)}{n|H|} \\
                                    &= n\rho(H)
\end{align*}

where, additionally, if $H$ is abelian, then $\rho(H) = 1$ and the conclusion follows.
\end{proof}

When at least two members of a TPP triple of a group are contained in a single larger subgroup the TPP ratio of the group is bounded by the TPP ratio of the subgroup, as described in the following result.

\begin{proposition}\label{PropTPPSplit} Let $G$ be a group, $H$ a subgroup, and $(S, T, U)$ a TPP triple of $G$. If $S, T \subseteq H$ then $\frac{|S||T||U|}{|G|} \leq \rho(H)$. Additionally, if $(S, T, U)$ is maximal for $G$ then $\rho(G) \leq \rho(H)$, and, additionally, if $H$ is abelian, then $\rho(G) = 1$. There is a subgroup version of this where "TPP triple" is replaced by "subgroup TPP triple", $\rho(G)$ by $\rho_0(G)$, and $\rho(H)$ by $\rho_0(H)$.
\end{proposition}
\begin{proof}Let $G$, $H$ and the TPP triple $(S, T, U)$ be given as above. Let $G/H$ be the collection of $n = |G:H|$ left cosets of $H$ in $G$, and let $\overline{U} = \{gH \in G/H \mathrel | U \cap gH \neq \emptyset\}$ be the $H$-support of $U$. As $1 \in U \cap H$ then $U$ is at least partly contained in $H$ and $U_0 = U \cap H$ is non-empty. Then $(S, T, U_0)$ forms a TPP triple of $H$ such that $\frac{|S||T||U_0|}{|H|} \leq \rho(H)$, and, by translation invariance \citep[Observation 2.1]{Neu2011} it may be supposed that any support element $gH \in \overline{U}$ intersects with $U$ in at most $|U_0|$ elements, that is, $|U \cap gH| \leq |U_0|$ for all $gH \in \overline{U}$. The collection $\overline{U}$ forms a partition of $U$ over the cosets of $H$ in $G$, and $|U| \leq n|U_0|$. Letting $U_1 = U \backslash U_0$ and $|U_1| = |U| - |U_0|$, so that $|U| = |U_0| + |U_1|$, where $|U_1| \leq (n - 1)|U_0|$, then

\begin{align*}\frac{|S||T||U|}{|G|} &= \frac{|S||T|(|U_0| + |U_1|)}{|G|} \\
                                    &= \frac{|S||T|(|U_0| + |U_1|)}{n|H|} \\
                                    &= \frac{1}{n}\frac{|S||T||U_0|}{|H|} + \frac{1}{n}\frac{|S||T||U_1|}{|H|} \\
                                    &\leq \frac{1}{n}\frac{|S||T||U_0|}{|H|} + \frac{n - 1}{n}\frac{|S||T||U_0|}{|H|} \\
                                    &= \frac{|S||T||U_0|}{|H|} \\
                                    &\leq \rho(H).
\end{align*}

The second part of the claim follows by taking $(S, T, U)$ to be maximal for $G$, that is, such that $\rho(G) = \frac{|S||T||U|}{|G|}$. From this the third part follows by taking $H$ to be abelian, in which case $\rho(H) = \rho(G) = 1$. The equivalent statement for subgroups follows by replacing "TPP triple" by "subgroup TPP triple", $\rho(G)$ by $\rho_0(G)$, and $\rho(H)$ by $\rho_0(H)$.
\end{proof}

There is a corollary - also independently implied by Proposition \ref{PropTPPPermSubsets} - that is not needed for the main results but is useful to note.

\begin{corollary} If $(S, T, U)$ is a non-trivial TPP triple of a nonabelian group $G$, that is, such that $|S||T||U| > |G|$, then no two members $X, Y \in \{S, T, U\}$ generate an abelian subgroup, that is, the subgroups $\langle S, T \rangle$, $\langle S, U \rangle$, $\langle T, U \rangle$ are all nonabelian.
\end{corollary}
\begin{proof} Use Proposition \ref{PropTPPSplit} and let $H = \langle X, Y \rangle$ for any two $X, Y \in \{S, T, U\}$.
\end{proof}

A basic fact about abelian normal subgroups in groups of nilpotency class $2$ realised as direct products will be useful to note.

\begin{observation}\label{ObsNilpClassTwo} If $G$ is a group of nilpotency class $2$, with a subgroup $H \leq G$ such that $H \cap Z(G) = \{1\}$, then $HZ(G)$ is an abelian normal subgroup of $G$ and a direct product, and, additionally, if $G/Z(G)$ is elementary abelian then so is $H$.
\end{observation}
\begin{proof}Let $G$ and $H$ be given as above, and let $H \cap Z(G) = \{1\}$. As $G' \leq Z(G)$ by assumption, $H' = [H, H] \leq H \cap G' = \{1\}$ and $H$ is abelian. As $Z(G)$ is normal in $G$ the product $HZ(G)$ is a subgroup of $G$. The quotient $G/Z(G)$ is abelian, $HZ(G)/Z(G)$ is normal in $G/Z(G)$ and, therefore, $HZ(G)$ is normal in $G$ (by the Correspondence Theorem between the subgroups of $G$ containing $Z(G)$ and the subgroups of $G/Z(G)$). As $H \cap Z(G) = \{1\}$ then $HZ(G)$ is an abelian direct product of order $|HZ(G)| = |H||Z(G)|$. Finally, if, additionally, $G/Z(G)$ is elementary abelian, say, $G/Z(G) \cong C_p^{\times n}$ for some positive integer $n$ and prime $p$, then $h^p \in Z(G)$ for all $h \in H$ implies that $h^p = 1$ for all $h \in H$ and $H$ is elementary abelian.
\end{proof}

The following is a useful result \citep{Neu2012} and a kind of conditional converse to \citep[Lemma 2.2]{CU}.

\begin{lemma}\textup{\citep{Neu2012}}\label{LemNeuQuoSubTPP} Let $(S, T, U)$ be a subgroup TPP triple of $G$ and $N$ a normal subgroup of $G$ such that $N \subseteq S$. Then the quotient $G/N$ realises the subgroup TPP triple $(S/N,TN/N,UN/N)$ of type $\left(|S|/|N|, |T|, |U|\right)$, and $\frac{|S||T||U|}{|G|} \leq \rho_0(G/N)$. In particular, if $(S, T, U)$ is maximal for $G$ then $\rho_0(G) \leq \rho_0(G/N)$.
\end{lemma}
\begin{proof} Recall that $S, T, U$, as TPP subgroups, satisfy $S \mathrel \cap TU = T \mathrel \cap U = \{1\}$. As $N \subseteq S$ by assumption then $T \cap N = U \cap N = TU \cap N =\{1\}$. 

\vspace{1mm}

Let $\pi:G \rightarrow G/N$ be the canonical projection of $G$ onto $G/N$, that is, $\pi(g) = gN$ for each $g \in G$, and let $\overline{S}, \overline{T}, \overline{U} \subseteq G/N$ be the images of $S, T, U$ respectively, under $\pi$, in $G/N$, that is, $\overline{S} = \pi(S) = \{sN \mathrel | s \in S\}$, and $\overline{T} = \pi(T) = \{tN \mathrel | t \in T\}$, and $\overline{U} = \pi(U) = \{uN \mathrel | u \in U \}$. Then $\overline{S} = SN/N = S/N$, and $\overline{T} = TN/N \cong T/\{1\} \cong T$, and $\overline{U} = UN/N \cong U/\{1\} \cong U$, and these are all subgroups of $G/N$.

\vspace{1mm}

Letting $1_{G/N} = N$ be the identity of $G/N$ then $\overline{T} \cap \overline{U} \cong T \cap U = \{1\}$, which implies that $\overline{T} \cap \overline{U} = \{1_{G/N}\}$. Consider $\overline{S} \cap \overline{T}\cdot \overline{U} = S/N \cap (TN/N \cdot UN/N) = S/N \cap (TU)N/N$, where $(TU)N/N = \{tuN \mathrel | t \in T, u \in U\}$ is a subset of $G/N$ and not necessarily a subgroup (unless $G/N$ is abelian). Let $xN \in S/N \cap (TU)N/N$. Then $xN = sN = tuN$ for some $s \in S$ and $tu \in TU$ for some $t \in T$ and $u \in U$, which implies that $tu \in sN \cap TU \subseteq S \cap TU = \{1\}$, that is, $tu = 1$, which, by the TPP for $(S, T, U)$, implies that $s = t = u = 1$. This shows that $S/N \cap (TU)N/N = (S \cap T)N/N = N/N \cong \{1\}$ and $\overline{S} \cap \overline{T}\cdot \overline{U} = \{1_{G/N}\}$, where $\overline{T}\cdot\overline{U} \cong TU$ is a set bijection, and that $(\overline{S}, \overline{T}, \overline{U})$ is a subgroup TPP triple of $G/N$ of type $(|S|/|N|, |T|, |U|)$ as claimed.

\vspace{1mm}

The last part of the claim follows by taking $(S, T, U)$ to be a maximal subgroup TPP triple of $G$.
\end{proof}

This has two simple corollaries.

\begin{corollary}\label{Corr1LemNeuQuoSubTPP} Any group $G$ contains a normal subgroup $N$ such that $\rho_0(G) \leq \rho_0(G/N)$.
\end{corollary}
\begin{proof} Let $(S, T, U)$ be a subgroup TPP triple of $G$ such that $\rho_0(G) = \frac{|S||T||U|}{|G|}$ and $N = \text{Core}_G(X) = \bigcap_{g \in G}g^{-1}Xg$ the normal $G$-core of any member $X \in \{S, T, U\}$, which is the largest normal subgroup of $G$ contained in $X$, formed by the intersection of all the conjugates of $X$ by elements of $G$: any member, $X = S$, say, may be chosen, in which case $N = \text{Core}_G(S)$ and by Lemma \ref{LemNeuQuoSubTPP} $G/N$ realises the subgroup TPP triple $(S/N,TN/N, UN/N)$ of size $(|S||T||U|)/|N|$, and $\rho_0(G/N) \geq \frac{(|S||T||U|)/|N|}{|G|/|N|} = \rho_0(G)$. For the tightest bound $X \in \{S, T, U\}$ can be chosen to maximise the order of $N = \text{Core}_G(X)$.
\end{proof}

Note above that if $N = \text{Core}_G(S) = \{1\}$ then $S$ is a core-free subgroup and $G$ is isomorphic to a transitive subgroup of $\text{Sym}(n)$ where $n = |G:S|$.

\begin{corollary}\label{Corr2LemNeuQuoSubTPP} Let $G$ be a nonabelian group and $(S, T, U)$ a non-trivial subgroup TPP triple of $G$, that is, $S, T, U$ are subgroups of $G$ satisying the TPP such that $|S||T||U| > |G|$. Then 

\vspace{1mm}

\textup{(1)} No member $S, T, U$ contains the commutator subgroup $G'$.

\vspace{1mm}

\textup{(2)} No member $S, T, U$ contains a normal subgroup $N$ of $G$ such that the quotient $G/N$ is abelian.

\vspace{1mm}

\textup{(3)} If $G$ is of nilpotency class $2$ then no member $S, T, U$ contains the centre $Z(G)$.

\vspace{1mm}

\textup{(4)} If $G$ is a $p$-group then no member $S, T, U$ contains the Frattini subgroup $\Phi(G)$.
\end{corollary}
\begin{proof} Let $G$ and $(S, T, U)$ be as given.

\vspace{1mm}

\textup{(1)} By Theorem \ref{ThmRhoAbelianGroups}) $\rho_0(H) = \rho(H) = 1$ if $H$ is abelian. The commutator subgroup $G'$ is the minimal normal subgroup of $G$ such that the quotient $G/G'$ is abelian, which means that $\rho_0(G/G') = 1$. Choosing any member, $S$, say, such that $G' \subseteq S$, by Lemma \ref{LemNeuQuoSubTPP} $G/G'$ realises the subgroup TPP triple $(S/G', TG'/G', UG'/G')$ of type $(|S|/|G'|, |T|, |U|)$, and $\frac{|S||T||U|}{|G|} = \frac{(|S||T||U|)/|G'|}{|G|/|G'|} \leq \rho_0(G/G') = 1$, a contradiction. So $S$, and therefore also $T$ and $U$, cannot contain $G'$, from the assumption that $S, T, U$ satisfy the TPP in $G$ such that $|S||T||U| > |G|$.

\vspace{1mm}

\textup{(2)} Follows from \textup{(1)}.

\vspace{1mm}

\textup{(3)} If $G$ is of nilpotency class $2$ then the centre $Z(G)$ contains the commutator subgroup $G'$ and the conclusion follows from \textup{(1)}.

\vspace{1mm}

\textup{(4)} If $G$ is a (nonabelian) $p$-group then the Frattini subgroup $\Phi(G)$ contains a non-trivial commutator subgroup $G'$, and the conclusion follows from \textup{(1)}.
\end{proof}

The following lemma relates the subgroup TPP ratio of a quotient of a group, by some normal subgroup, to the index of its centre, and will be useful.

\begin{lemma}\label{LemQuoNilpClassTwo} If $G$ is a group with a normal subgroup $N$ such that $\rho_0(G/N) \leq \sqrt{G/N:Z(G/N)}$ then $\rho_0(G/N) \leq \sqrt{G:Z(G)}$.
\end{lemma}
\begin{proof} Let $G$, $N$ and $\rho_0(G/N)$ be given as above. Note that $Z(G)N/N \leq Z(G/N)$, where $Z(G)N/N \cong Z(G)/(Z(G) \cap N)$ and $|Z(G)N/N| = \frac{|Z(G)|}{|Z(G) \cap N|}$, and therefore $\frac{1}{|Z(G/N)|} \leq \frac{1}{|Z(G)N/N|} = \frac{|Z(G) \cap N|}{|Z(G)|}$. Now
\begin{align*}\rho_0(G/N) &\leq \sqrt{G/N:Z(G/N)} = \sqrt{\frac{|G/N|}{|Z(G/N)|}} \\
						  &=  \sqrt{\frac{|G|}{|N|}\cdot\frac{1}{|Z(G/N)|}} = \frac{\sqrt{|G|}}{\sqrt{|N|}}\cdot\frac{1}{\sqrt{|Z(G/N)|}}\\
                          &\leq \frac{\sqrt{|G|}}{\sqrt{|N|}}\cdot\frac{1}{\sqrt{|Z(G)N/N|}} \\
                          &= \frac{\sqrt{|G|}}{\sqrt{|N|}}\cdot\frac{\sqrt{|Z(G) \cap N|}}{\sqrt{|Z(G)|}} \\
                          &\leq \frac{\sqrt{|G|}}{\sqrt{|N|}}\cdot\frac{\sqrt{|N|}}{\sqrt{|Z(G)|}} \\
                          &= \sqrt{G:Z(G)}.
\end{align*}
\end{proof}

Two commutator identities for groups of nilpotency class $2$ will be used, one following from the other. The first is

\begin{equation}\label{EqCommId} [xy, z] = [x, z][y, z] \hspace{3em}x, y, z \in G \text{ (where $G$ is of nilpotency class $2$)}
\end{equation}

This identity, which has the "inverse" identity $[z, xy] = [z, x][x, y]$, uses the fact that, in this case, commutators are central and for elements $x, y, z \in G$ the identity is given by $[xy, z] = (xy)^{-1}z^{-1}xyz = y^{-1}x^{-1}z^{-1}xyz = y^{-1}x^{-1}z^{-1}xzz^{-1}yz = y^{-1}[x, z]z^{-1}yz = [x, z]y^{-1}z^{-1}yz = [x, z][y, z]$. From this follows the second identity:

\begin{equation}\label{EqCommId2} [x^p, y] = [x, y]^p.
\end{equation}

For the latter note that by \ref{EqCommId} $[x^p, y] = [x\cdot x^{p - 1}, y] = [x, y][x^{p - 1}, y]$, and repeated application leads to $[x, y]^p$, which can be proved by induction.

\vspace{1mm}

The following is a proposition that will frequently be useful in some of the later results for $p$-groups of class $2$.

\begin{proposition}\label{PropRho0p6} If $G$ is a $p$-group of order $|G| \leq p^6$ then $\rho_0(G) \leq p$, and, in particular, $\rho_0(G) = 1$ if $|G| \leq p^4$.
\end{proposition}
\begin{proof} All $p$-groups of order $\leq p^2$ are abelian and realise $\rho_0 = \rho = 1$. For nonabelian $p$-groups of order $p^3$ and $p^4$ the result follows from combinatorial constraints on possible parameters of subgroup TPP triples imposed by the inequalities in \citep[Observation 3.1]{Neu2011}, namely the types $(p, p, p)$ and $(p^2, p, p)$ respectively, in both cases implying $\rho_0 = 1$. Some simple group-theoretic arguments also exist here. Nonabelian $p$-groups of order $p^3$, which are extraspecial, and those of order $p^4$, contain abelian normal (and maximal) subgroups of index $p$ (see \citep[Lemma 3.4]{Cra} for $p$-groups of order $p^3$, and for $p$-groups of order $p^4$ see \citep[Lemma 2]{Wil} for $p = 2$, and \citep[Proposition 12]{AGW} for $p > 2$). So for these groups $\rho_0 = 1$ by \citep[Corollary 4.3]{Mur2025}.

\vspace{1mm}

When $|G| = p^5$ or $|G| = p^6$ \citep[Observation 3.1]{Neu2011} can be used once again to deduce that $(p^2, p^2, p^2)$ and $(p^3, p^2, p^2)$ respectively are the only possible parameter types of maximal subgroup TPP triples that achieve $\rho_0 > 1$, in both cases implying that $\rho_0 = p$.
\end{proof}

For $p$-groups with a cyclic commutator subgroup of order $p$, including extraspecial groups, or $p$-groups of nilpotency class $2$ that satisfy a large centre assumption $G:Z(G) \leq p^3$, some basic results are required, all of which are known or can be deduced from the basic theory of (finite) $p$-groups.

\begin{lemma}\label{LemCycNormSubOrderp} If $N$ is a normal subgroup of a nonabelian $p$-group $G$ such that $N$ is cyclic of order $p$ then $N$ is central, that is, $N \trianglelefteq G$ and $N \cong C_p$ implies that $N \leq Z(G)$. In particular, if $G$ is a (nonabelian) $p$-group containing a cyclic commutator subgroup $G'$ of order $p$ then $G'$ is central and $G$ is of nilpotency class $2$.
\end{lemma}
\begin{proof} Let $G$ and $N$ be as given. From the natural action of $G$ on the elements of $N$ by conjugation, given by $n^g \longmapsto g^{-1}ng = n'$ for elements $n \in N$ and $g \in G$, where $n' \in N$ depends on $n$, there are conjugacy classes $n^G = \{g^{-1}ng \mathrel | g \in G\}$ for $n \in N$, that partition $N$. The number $k$ of these conjugacy classes satisfies $1 \leq k \leq p = |N|$, as there is a trivial singleton conjugacy class $1^G = \{1\}$, and each conjugacy class size $|n^G|$ must divide $|G|$ and is therefore either $1$ or $p$. This implies that every non-identity element of $N$ also forms a singleton conjugacy class, that is, all elements of $N$ are self-conjugate: $g^{-1}ng = n$ for any $n \in N$ and $g \in G$. This is equivalent to $[n, g] = n^{-1}g^{-1}ng = 1$ for all $g \in G$ and $n \in G$, that is, all elements of $N$ commute with all elements of $G$, formally, $[N, G] = \{1\}$, which proves that $N \leq Z(G)$. The second part of the claim follows by substituting $G'$ for $N$, in which case there is a central series for $G$ given by $1 < G' \leq Z(G) < G$ of length $2$.
\end{proof}

\begin{lemma}\label{LemNonAbCentQuo} If $G$ is a nonabelian group then $G/Z(G)$ cannot be cyclic.
\end{lemma}
\begin{proof}
Suppose $G$ is nonabelian such that $G/Z(G)$ is cyclic, say, of order $n = G:Z(G)$. Then $G/Z(G) = \langle g_0Z(G) \rangle = \{(g_0Z(G))^i = g_0^iZ(G) \mathrel | i \in \mathbb{Z}_n\}$ for some $g_0 \in G\backslash\{1\}$, and $G = \bigcup_{i \in \mathbb{Z}_n}g_0^iZ(G)$, which means that any element $x \in G$ can be written as $g_0^iz$ for some $i \in \mathbb{Z}_n$ and $z = z(x) \in Z(G)$ depending on the element $x$. If $x, y \in G$ are any two elements then $x = g_0^iz_x$ and $y = g_0^jz_y$ for some $i, j \in \mathbb{Z}_n$ and $z_x, z_y \in Z(G)$, and $xy = g_0^iz_x \cdot g_0^jz_y = g_0^ig_0^jz_xz_y = g_0^{i + j}z_xz_y = g_0^jg_0^iz_yz_x = g_0^jz_y \cdot g_0^iz_x = yx$, which implies that $G$ is abelian, a contradiction.
\end{proof}

\begin{proposition}\label{PropAbCentQuop2p3} If $G$ is a (non-abelian) $p$-group such that either $G/Z(G) \cong C_p \times C_p$ or $G/Z(G) \cong C_{p^2} \times C_p$, then $G$ contains abelian normal (and maximal) subgroups of index $p$.
\end{proposition}
\begin{proof} Let $G$ be as given. Then in both cases $G/Z(G)$ is abelian and $G' \leq Z(G)$, which means that $G$ is of nilpotency class $2$. 

\vspace{1mm}

\textup{(i)} Let $G/Z(G) \cong C_p \times C_p$, i.e. $G/Z(G)$ is elementary abelian of order $p^2$. There is a cyclic (and maximal) normal subgroup $L/Z(G) \leq G/Z(G)$, say, of index $p$, and isomorphic to $C_p \times \{1\}$, that is generated by an element $xZ(G) \in G/Z(G)$ of order $p$, where $x \in G \backslash Z(G)$ is of order $\geq p$. By the Correspondence Theorem there is a unique subgroup $L \leq G$ corresponding to $L/Z(G)$, containing $Z(G)$, which is an abelian normal (and maximal) subgroup of index $p$, and can be taken to be $L = \langle Z(G), x \rangle$, where $L = C_G(x)$, that is, $L$ is the centraliser of $x$ in $G$.

\vspace{1mm}

\textup{(ii)} Let $G/Z(G) \cong C_{p^2} \times C_p$. Then $G/Z(G)$ contains a subgroup $L/Z(G)$, say, of index $p$, which is abelian, maximal and normal in $G/Z(G)$, and by the Correspondence Theorem there is a unique subgroup $L \leq G$ of index $p$, containing $Z(G)$ and corresponding to $L/Z(G)$, which is abelian, maximal and normal in $G$. Then $L/Z(G) \cong C_{p^2} \times \{1\}$, which means that $L/Z(G)$ is generated by some element $xZ(G) \in G / Z(G)$ of order $p^2$, where $x \in G \backslash Z(G)$ is of order $\geq p^2$, and $L$ can be taken to be $L = C_G(x) = \langle Z(G), x \rangle$.
\end{proof}

Note that the conclusion in Proposition \ref{PropAbCentQuop2p3} fails for the case when $G/Z(G) \cong C_p \times C_p \times C_p$ as there are counter-examples. One such counter-example, found with GAP, is when $G$ is a $2$-group of order $64$ and type $(C_2 \times C_2 \times D_8) \rtimes C_2$ satisfying $G/Z(G) \cong C_2 \times C_2 \times C_2$ (GAP ID \text{[64, 73]}), in which all subgroups of index $2$ are non-abelian.

\vspace{1mm}

The main result for subgroup TPP ratio of extraspecial groups is Corollary \ref{CorrRho0pGroupCycComm}, which derives from Theorem \ref{ThmRho0pGroupCycComm}, and these require some preliminary results that will be proved here. The main references for extraspecial groups are \citep{Cra} and \citep{Ber}. First, some basic definitions and properties of extraspecial groups of relevance are recalled. First, an extraspecial group is a (non-abelian) $p$-group $G$ such that $G' = Z(G) = \Phi(G) \cong C_p$, so that $G/Z(G)$ is elementary abelian and $G$ is of nilpotency class $2$. Equivalently, an extraspecial group can be defined from so-called \emph{special} groups $G$, which are defined as $p$-groups that are either elementary abelian, or such that $G$ is of nilpotency class $2$ and $G'$ is elementary abelian such that $G' = \Phi(G) = Z(G)$; from which an extraspecial group can be defined as a non-abelian special group $G$ such that $|G'| = p$.

\vspace{1mm}

For $p = 2$ there are two extraspecial groups of order $2^{1+2} = 8$, namely, the dihedral group $D_8 = \langle x, y \mid x^4 = y^2 = 1, [x, y] = x^2 \rangle$, and the quaternion group $Q_8 = \langle x, y \mid x^4 = y^4 = 1, [x, y] = x^2, x^2 = y^2 \rangle$, and both have exponent $4$ as they have elements of order $4$. For $p > 2$ there are two extraspecial groups of order $p^3$, namely, one denoted by $p_+^{1+2}$ of exponent $p$ and the presentation:

\begin{equation} p_+^{1+2} = \langle x, y, z \mid x^p = y^p = z^p = 1, [x, z] = [y, z] = 1, [x, y] = z \rangle
\end{equation}

and the other denoted by $p_-^{1+2}$ of exponent $p^2$ and the presentation:

\begin{equation} p_-^{1+2} = \langle x, y, z \mid x^p = y^p = z, z^p = 1, [x, z] = [y, z] = 1, [x, y] = z \rangle
\end{equation}

Of particular relevance to $\rho_0$ is that all extraspecial groups of order $p^3$, for any prime $p$, contain abelian normal (and maximal) subgroups of index $p$, which means that they satisfy $\rho_0 = 1$ by \citep{Mur2025}[Corollary 4.3]. For $p = 2$ general extraspecial groups of order $2^{1+2n}$ are $n$-fold central products $D_8^{*n}$ or $Q_8 * D_8^{*(n - 1)}$, while for $p > 2$ they are $n$-fold central products $G_1 * \cdots * G_n$ of extraspecial groups $G_1,\ldots,G_n$ of type either $p_+^{1+2}$ or $p_-^{1+2}$. Recall that a central product $G * H$ of two groups $G$ and $H$ exists when, given an isomorphism $\theta: Z(G) \cong Z(H)$ between their centres, the subset $G * H = (G \times H) / N \subseteq G \times H$, with $N = \{(z, \theta(z)^{-1}) \mid z \in Z(H)\}$ a normal subgroup of $G \times H$ of order $|Z(G)| = |Z(H)|$, forms a group of order $|G \times H| / |Z| = |G||H|/|Z|$, where $Z = Z(G * H) = Z(G) \cong Z(H)$, and is called the central product of $G$ and $H$ with respect to $\theta$. For each isomorphism $\theta$ between $Z(G)$ and $Z(H)$ there is a unique central product with respect to $\theta$. If $|Z(G)| = |Z(H)| = p$ then there is only one such isomorphism and therefore only one central product $G * H$.

\begin{proposition}\label{Exspp5SubTPP} Let $G$ be extraspecial of order $p^{1+4}$ such that either \textup{(1)} $p = 2$ and $G = D_8^{*2} = D_8 * D_8$, where $G$ is extraspecial of order $2^{1+4} = 32$ and exponent $4$, or \textup{(2)} $p > 2$ and $G = (p_+^{1+2})^{*2}$ is extraspecial of order $p^{1+4}$ and exponent $p$, where in either case $G$ has the central product presentation:
$$
    G = \left\langle x_1,y_1,x_2,y_2,z\ \middle|\ 
            \begin{aligned}
                &x_i^p = y_i^p = z^p = 1, \qquad [x_i,y_i] = z && (i = 1,2)\\
                &[x_i,x_j] = [x_i,y_j] = [y_i,y_j] = 1 && (i \neq j)\\
                &[x_i,z] = [y_i,z] = 1 && (i = 1,2)
            \end{aligned}
        \right\rangle
$$

Then $G$ contains the following three elementary abelian subgroups of order $p^2$:

\begin{align*} S := \langle x_1, x_2 \rangle, \ T := \langle y_1, y_2 \rangle, \ U := \langle zx_1y_2, zx_1x_2y_1y_2 \rangle 
\end{align*}

that satisfy the TPP and $\rho_0(G) \geq p$.
\end{proposition}
\begin{proof} \textup{(1)} Let $p = 2$ and $G = D_8^{*2}$. Note that if $G_1 := \langle x_1, y_1 \rangle \cong D_8$ and $G_2 := \langle x_2, y_2 \rangle \cong D_8$, then $G_i' = [G_i, G_i] = G' = Z(G) = \langle z \rangle \cong C_2$ and $[G_i, G_j] = \{1\}$ for $i \in \{1, 2\}$ such that $i \neq j$ and $G = G_1 * G_2$ is their central product. Also note that elements of the form $x_iy_i$ satisfy $(x_iy_i)^2 = z$, and for  $i \neq j$ there is also an identity $(x_iy_j)^2 = x_iy_jx_iy_j = x_i(x_iy_j)[y_j, x_i]y_j = x_i^2y_j^2 = 1$.

\vspace{1mm}

Let $S, T, U$ be as given. Clearly $S$ and $T$ are elementary abelian. The elements of $S$ and $T$ are of two distinct forms $x_1^ax_2^b$ and $y_1^cy_2^d$, for $a, b, c, d \in GF(2)$, respectively, hence $S \cap T = \{1\}$, and they are also all involutions as $(x_1^ax_2^b)^2 = (x_1^a)^2(x_2^b)^2 = (x_1^2)^a(x_2^2)^b = 1$ from the commutativity of the $x_i$, and analogously $(y_1^cy_2^d)^2 = 1$. The two generators of $U$ are both involutions as $(zx_1y_2)^2 = zx_1y_2zx_1y_2 = z^2x_1^2y_2^2 = 1$ and $(zx_1x_2y_1y_2)^2 = z^2(x_1x_2y_1y_2)^2 = (x_1y_1x_2y_2)^2 = x_1y_1x_2y_2x_1y_1x_2y_2 = z^2x_1^2y_1^2x_2^2y_2^2 = 1$, and they commute as can be verified using a simple commutator calculation using \ref{EqCommId}: $[zx_1y_2, zx_1x_2y_1y_2] = [x_1y_2,x_1x_2y_1y_2] = z^2 = 1$, thus showing $U$ is elementary abelian. Elements of the set product $ST$ are of the form $x_1^ax_2^by_1^cy_2^d$. Suppose $x_1^ax_2^by_1^cy_2^d = zx_1y_2 = x_1y_2z$ for some $a, b, c, d \in GF(2)$. Rewriting $x_1^ax_2^by_1^cy_2^d = x_1^ay_2^dx_2^by_1^c$ the equation above becomes $x_1^ay_2^dx_2^by_1^c = x_1y_2z$, which implies that $a = d = 1$ and $z = x_2^by_1^c$ for $b, c \in GF(2)$, which is impossible. Suppose $zx_1x_2y_1y_2 = zx_1y_2$ for some $a, b, c, d \in GF(2)$, which implies that $x_1x_2y_1y_2 = x_1y_2x_2y_1 = x_1y_2$, and therefore that $x_2y_1 = 1$ (by cancellation), which is impossible. This means that no non-identity element $x_1^ax_2^by_1^cy_2^d \in ST$ can be represented by a generator of $U$, or, therefore, as $U$ is abelian, by an arbitrary element of $U$.  That is, $ST \cap U = \{1\}$. Therefore $(S, T, U)$ is a TPP triple of type $(4, 4, 4)$ and size $|S||T||U| = 4^3 = 64$ and $\rho_0(G) \geq \frac{64}{32} = 2$.

\textup{(2)} There is an analogous version of \textup{(1)} where $p = 2$ is replaced by $p > 2$, $G$ is extraspecial of type $(p_+^{1+2})^{*2}$ and exponent $p$, and involutions are replaced by elements of order $p$, where the resulting TPP triple is of type $(p^2, p^2, p^2)$ and size $p^6$, and $\rho_0(G) \geq p$.

\end{proof}

\begin{proposition}\label{ExspGroupBilForm} Let $G$ be an extraspecial group of order $p^{1 + 2n}$ and $G/Z$ its quotient by $Z = Z(G)$ of order $p^{2n}$.

\vspace{1mm}

\textup{(1)} There is a non-singular (or non-degenerate) alternating bilinear form $\phi: V \times V \longrightarrow GF(p)$ of rank $2n$, where $V = G/Z$ is viewed as a $2n$-dimensional vector space over $GF(p)$, given by the mapping $(\bar{x}, \bar{y}) \longmapsto [x, y] = z^{i(x, y)}$ for elements $\bar{x} = xZ(G), \bar{y} = yZ(G) \in V$, where $i(x, y) \in GF(p)$ is identified with $z^{i(x, y)} \in Z$ and $z$ is a generator of $Z$.

\vspace{1mm}

\textup{(2)} If $S, T \leq G$ are elementary abelian subgroups of order $p^n$ such that $S \cap T = S \cap Z(G) = T \cap Z(G) = \{1\}$ then for any non-zero subspace $W \leq V$ there is a quadratic form $Q: W \longrightarrow GF(p)$ given by $Q(w) = i(s, t)$ for elements $w \in W$ where $w$ is uniquely determined by elements $s = s(w) \in S$ and $t = t(w) \in T$, and $i(s, t)$ corresponds uniquely with $z^{i(s, t)} = \phi(\bar{s}, \bar{t}) = [s, t]$ where $z$ is a generator of $Z(G)$.

\end{proposition}
\begin{proof} Let $G$ be as given, and let $Z = Z(G) = \Phi(G) = G' \cong C_p$.  Then $G/Z$ is an elementary abelian group of order $p^{2n}$, and $V = G/Z$ can be viewed as a $2n$-dimensional vector space over $GF(p)$, where the zero vector $0 \equiv 0_V$ corresponds to the quotient identity $1_{G/Z} \equiv Z$. For an element $x \in G$ its coset image $xZ \in G/Z$, considered as a vector in $V$, is denoted by $\bar{x}$. Coset multiplication in $G/Z$ corresponds to vector addition in $V$, commutators in $G$ are elements of $Z$, and multiplication of commutators in $G$ corresponds to scalar addition in $GF(p)$. Let $z \in Z$ be a generator and $\phi$ a map $\phi: V \times V \longrightarrow GF(p)$ given by $\phi(\bar{x}, \bar{y}) \longmapsto i = i(x, y)$, for $\bar{x}, \bar{y} \in V$, where $i \in GF(p)$ is uniquely determined by $[x, y] = z^i \in Z$, so that $i$ can be identified with $z^i$. Note that the commutator identities (\ref{EqCommId}) and (\ref{EqCommId2}) apply here as $G$ is of nilpotency class $2$.

\vspace{1mm}

\text{(1)} For elements $a, b \in GF(p)$ and $\bar{x}, \bar{y} \in V$ the elements $a\bar{x}, b\bar{y} \in V$ correspond to the cosets $x^aZ = (xZ)^a$ and $y^bZ = (yZ)^b$ respectively, and their sum $a\bar{x} + b\bar{y}$ corresponds to the coset product $x^aZ \cdot y^bZ = x^ay^bZ$. Then for $\bar{w} \in V$ we have $\phi(a\bar{x} + b\bar{y}, \bar{w}) = [x^ay^b, w] = z^{i(x^ay^b, w)} = [x^a, w][y^b, w] = [x, w]^a[y, w]^b = \left(z^{i(x, w)}\right)^a \left(z^{i(y, w)}\right)^b = z^{ai(x, w)}z^{bi(y, w)} = z^{ai(x, w) + bi(y, w)} \equiv a\phi(\bar{x}, \bar{w}) + b\phi(\bar{y}, \bar{w})$, which shows that $\phi$ is linear in the first component. Analogously, $\phi$ is also linear in the second component as $\phi(\bar{x}, a\bar{y} + b\bar{w}) = [x, y^aw^b] = z^{i(x, y^aw^b)} = [x, y^a][x, w^b] = [x, y]^a[x, w]^b = \left(z^{i(x, y)}\right)^a\left(z^{i(x, w)}\right)^b = z^{ai(x, y)}z^{bi(x, w)} = z^{ai(x, y) + bi(x, w)} \equiv a\phi(\bar{x}, \bar{y}) + b\phi(\bar{x}, \bar{w})$. \text{(2)} As $[x, x] = 1$ for all $x \in G$ then $\phi(\bar{x}, \bar{x}) = 0$  for all $\bar{x} \in V$, which shows that $\phi$ is an alternating form. \text{(3)} The form $\phi$ is also non-singular (or non-degenerate) as for any $\bar{x} \in V$ such that $\phi(\bar{x}, \bar{y}) = 0$ for all $\bar{y} \in V$, the element $x \in G$ satisfies $[x, y] = 1$ for all $y \in G$, that is, $x \in Z$, which means that $\bar{x} = 0$. \textup{(4)} Finally, the rank of $\phi$ is the dimension $2n$ of $V$ as a vector space over $GF(p)$.

\vspace{1mm}

\textup{(2)} Let the subgroups $S, T \leq G$ be as given. Then there are subspaces $\bar{S} = SZ/Z, \hspace{0.25em}\bar{T} = TZ/Z \leq V$ that are both $n$-dimensional, because the projection mapping $x \longmapsto \bar{x} = xZ \in V$ is injective and therefore bijective on both $S$ and $T$, and therefore also on $ST = \{st \mid s \in S, \hspace{0.25em} t \in T \} \subseteq G$, which is a subset of size $|ST| = |S||T| = p^{2n}$. From the latter it follows that $|\bar{S}||\bar{T}| = |S||T| = |V| = p^{2n}$, and $\bar{S}$ and $\bar{T}$ are complementary subspaces, i.e. $\bar{S} \cap \bar{T} = \{0\}$, and $V$ has a direct sum decomposition $V = \bar{S} \oplus \bar{T}$. Two further observations follow: firstly, $ST \cap Z = \{1\}$, as otherwise the projection of $ST$ onto $V$ would not be injective, and secondly,  as $|G| = p^{1+2n} = |ST||Z| = |S||T||Z|$ and $ST \subseteq \langle S, T \rangle \leq G$, then $\langle S, T \rangle = SZ(G)T = G$, where $SZ(G)T$ is a semidirect product $SZ(G) \rtimes T$ because $SZ(G)$ is normal, while $T$ is non-normal, in $G$.

\vspace{1mm}

Let $W \leq V$ be a subspace: any element $w \in W$ is a direct sum $w = \bar{s} + \bar{t}$ of elements $\bar{s} \in \bar{S}$ and $\bar{t} \in \bar{T}$ uniquely determined by elements $s \in S$ and $t \in T$ respectively. Let $Q$ be the map $Q: W \longrightarrow GF(p)$ given by $Q(w) = i(s, t) \in GF(p)$ for elements $w = \bar{s} + \bar{t} \in W$, where $i(s, t)$ is the unique element of $GF(p)$ such that $z^{i(s, t)} = \phi(\bar{s}, \bar{t}) = [s, t]$, elements $i(s, t) \in GF(p)$ are identified with $\phi(\bar{s}, \bar{t})$ and $Q(w)$ is written as $Q(w) = \phi(\bar{s}, \bar{t})$. Let $a \in GF(p)$ and $w = \bar{s} + \bar{t} \in W$: then $Q(aw) = Q(a\bar{s} + a\bar{t}) = \phi(a\bar{s}, a\bar{t}) = a^2\phi(\bar{s}, \bar{t}) = a^2Q(w)$. The polar form of $Q$ is given by $Q(w_1 + w_2) - Q(w_1) - Q(w_2)$ is a bilinear form: to see this, let $w_1 = \bar{s_1} + \bar{t_1}, \hspace{0.25em} w_2 = \bar{s_2} + \bar{s_2} \in W$, then $Q(w_1 + w_2) = Q((\bar{s_1} + \bar{s_2}) + (\bar{t_1} + \bar{t_2})) = \phi(\bar{s_1} + \bar{s_2}, \bar{t_1} + \bar{t_2}) = \phi(\bar{s_1}, \bar{t_1}) + \phi(\bar{s_1}, \bar{t_2}) + \phi(\bar{s_2}, \bar{t_1}) + \phi(\bar{s_2}, \bar{t_2})$, and $Q(w_1) = Q(\bar{s_1} + \bar{t_1}) = \phi(\bar{s_1}, \bar{t_1})$ and $Q(w_2) = Q(\bar{s_2} + \bar{t_2}) = \phi(\bar{s_2}, \bar{t_2})$, which means that $Q(w_1 + w_2) - Q(w_1) - Q(w_2) = \phi(\bar{s_1}, \bar{t_2}) + \phi(\bar{s_2}, \bar{t_1})$, which is a bilinear function of $w_1, w_2 \in W$.
\end{proof}

\begin{lemma}\label{ExspGroupsMainRho0} Let $G$ be an extraspecial group of order $p^{1+2n}$. If $S, T, U \leq G$ are elementary abelian subgroups of order $p^n$ satisfying the TPP then $n \leq 2$ and $(S, T, U)$ is a maximal subgroup TPP triple of $G$ through which $G$ achieves $\rho_0(G) = p$. Specifically, either \textup{(1)} $n = 1$, $|G| = p^3$, $|S| = |T| = |U| = p$ and $\rho_0 = 1$, or \textup{(2)} $n = 2$, $p = 2$, $G = D_8^{*2}$ and $|G| = 32$, $|S| = |T| = |U| = 2^2 = 4$ and $\rho_0 = 2$, or $n = 2$, $p > 2$, $G = (p_+^{1+2})^{*2}$ and $|G| = p^5$, $|S| = |T| = |U| = p^2$ and $\rho_0 = p$.
\end{lemma}
\begin{proof}
Let $G$ be as given. Let $Z = Z(G)$ and $V = G/Z$ the $2n$-dimensional vector space over $GF(p)$ as described in Proposition \ref{ExspGroupBilForm}. From the same Proposition we know that $V$ has a direct sum decomposition $V = \bar{S} \oplus \bar{T}$, where $\bar{S} = SZ/Z$ and $\bar{T} = TZ/Z$ are complementary $n$-dimensional subspaces of $V$, which, recall, are in bijection with $S$ and $T$ respectively.

\vspace{1mm}

Let $W = UZ/Z \leq V$, where this is an injective projection of $U$ onto $G/Z$ and an $n$-dimensional subspace of $V$, and let $Q$ be the quadratic form on $W$, as described in Proposition \ref{ExspGroupBilForm}, given by $Q(w) = i(s, t) \in GF(p)$ for elements $w = \bar{s} + \bar{t} \in W$, where $i(s, t)$ is the unique element of $GF(p)$ such that $z^{i(s, t)} = \phi(\bar{s}, \bar{t}) = [s, t]$. By definition $Q$ can be represented as a homogeneous polynomial $P_Q(w) = P_Q(x_1,\ldots,x_n)$ of degree $2$ over $GF(p)$. This polynomial $P_Q$ can be determined by, first, choosing a basis of $W$, say, $w_1,\ldots,w_n$. With respect to this basis any element $w \in W$ can be written as the sum $w = \sum_i x_iw_i$ where the $\{x_i\}$ are the coefficients of $w$ with respect to the basis $\{w_i\}$. From $V = \bar{S} \oplus \bar{T}$ each basis element $w_i$ is the sum $w_i = s_i + t_i$ where the $\{s_i\}$ and $\{t_i\}$ are chosen bases of $\bar{S}$ and $\bar{T}$ respectively. Let $a_{ij} = \phi(s_i, t_j)$ for $1 \leq i, j \leq n$. Treating the coefficients $\{x_i\}$ as unknowns in $GF(p)$ the polynomial $P_Q$ is given by:

\begin{align*} P_Q(w) &= Q\left(\sum_{1 \leq i \leq n} x_iw_i\right) = Q\left(\sum_{1 \leq i \leq n} (x_is_i + x_it_i)\right) \\
                      &= \phi\left(\sum_{1 \leq i \leq n} x_is_i, \sum_{1 \leq j \leq n} x_jt_j\right) = \sum_{1 \leq i \leq n}\sum_{1 \leq j \leq n} \phi(s_i, t_j)x_ix_j \\
                      &= \sum_{1 \leq i, j \leq n} \phi(s_i, t_j)x_ix_j = \sum_{1 \leq i, j \leq n} a_{ij}x_ix_j
\end{align*}

As $\phi$ is alternating on $V$ it is also skew-symmetric, that is, $\phi(v, v') = -\phi(v', v)$ for any two elements $v, v' \in V$. This means that for any two basis elements $w_i = s_i + t_i, w_j = s_j + t_j$ it follows that $0 = \phi(w_i, w_j) = \phi(s_i + t_i, s_j + t_j) = \phi(s_i, t_j) + \phi(t_i, s_j) = \phi(s_i, t_j) - \phi(s_j, t_i)$, that is, $\phi(s_i, t_j) = \phi(s_j, t_i)$, or $a_{ij} = a_{ji}$, which leads to a further refinement of $P_Q$:

\begin{align*} P_Q(w) &= \sum_{1 \leq i, j \leq n} a_{ij}x_ix_j \\
                    &= \sum_{1 \leq i \leq n} a_{ii}x_i^2 + 2\left(\sum_{1 \leq i < j \leq n} a_{ij}x_ix_j\right)
\end{align*}

Clearly $P_Q$ is a homogeneous polynomial of degree $2$ over $GF(p)$, in $n$ variables $x_1,\ldots,x_n$. Note that if $p = 2$ then $2\sum_{1 \leq i < j \leq n} a_{ij}x_ix_j$ vanishes modulo $2$ and $P_Q(w) = \sum_{1 \leq i \leq n} a_{ii}x_i^2$, which is formally homogeneous of degree $2$ in the $x_i$, although because $x_i^2 \equiv x_i \hspace{0.5em} (\text{mod } 2)$ it can be reduced to linear form.

\vspace{1mm}

From part \textup{(2)} of Proposition \ref{ExspGroupBilForm} $G = STZ$, and every element $u \in U$ has a unique expression $u = stz^i$ for elements $s = s(u) \in S$, $t = t(u) \in T$ and $i = i(s, t) = [s, t] \in GF(p)$. If $w \in W$ then it has a unique lift $u = stz^i \in U$. The TPP condition $U \cap ST = \{1\}$ implies that $u \in ST$ holds if and only if $u = s = t = 1$, which in turn implies that $Q(w) = i = 0$ if and only if $w = 0$. In terms of the polynomial $P_Q$ this means that $P_Q = 0$ only for the trivial zero $(0,\ldots,0) \in GF(p^n)$. If $n > 2 = \text{deg}(P_Q)$ however, by the Chevalley-Warning theorem for polynomials over finite fields \citep{Hea} $P_Q$ has non-trivial zeros (in fact, whose number is a multiple of $p$), a contradiction. Thus $n \leq 2$.

\vspace{1mm}

If $n = 1$ then $G$ is extraspecial of order $p^3$ and $(S, T, U)$ is a maximal subgroup TPP triple of $G$ of type $(p, p, p)$ and $\rho_0(G) = 1$. If $n = 2$ then first note that by Proposition \ref{PropRho0p6} $\rho_0(G) \leq p$. There are two cases when $n = 2$, namely, $p = 2$ and $p > 2$. If $p = 2$ then $G = D_8^{*2}$ (GAP ID \text{[32, 49]}), which is extraspecial of order $2^{1+4} = 32$ and exponent $4$ and has $12$ elements of order $4$ \citep{Cra}[Proposition 3.13], and $(S, T, U)$ is a maximal subgroup TPP triple of type $(4, 4, 4)$ and $\rho_0(G) = 2$ as shown in Appendix A Table \ref{tbl:TblRho0GroupsOfTypeIIa}. Note that $G$ cannot be $D_8 * Q_8$ (GAP ID \text{[32, 50]}), which is the other extraspecial group of order $32$ and exponent $4$ but with $20$ elements of order $4$ \citep{Cra}[Proposition 3.13], for which $\rho_0(G) = 1$ as shown in Appendix A Table \ref{tbl:TblRho0GroupsOfTypeIIa}. If $p > 2$ then $G$ is extraspecial of type $(p_+^{1+2})^{*2}$ and exponent $p$, and $(S, T, U)$ is a maximal subgroup TPP triple of type $(p^2, p^2, p^2)$ of the kind described in Proposition \ref{Exspp5SubTPP} through which $G$ achieves $\rho_0(G) = p$. A specific example of the latter case, for $p = 3$, occurs in the extraspecial group of order $3^{1+4} = 243$ with exponent $3$ (GAP ID \text{[243, 65]}) via a TPP triple of type $(9, 9, 9)$ and composed of elementary abelian subgroups, through which it achieves $\rho_0 = 3$. Note that in the $p > 2$ case $G$ cannot be extraspecial of type $(p_-^{1+2})^{*2}$ and exponent $p^2$, as $G$ would not satisfy the condition $\langle S, T \rangle = G$ (part \textup{(1)} of Proposition \ref{ExspGroupBilForm}) as that would imply that $G$ has exponent $p$, given that $S$ and $T$ are elementary abelian.
\end{proof}

\section{Groups of Nilpotency Class $2$}\label{SecGroupsOfTypeI}

The subgroup TPP ratio data for groups of nilpotency class $2$ presented in Appendix A Table \ref{tbl:TblRho0GroupsOfTypeI} \citep[subset of Tables 1-4]{HM} suggests an inexact (or strict) upper bound of $\sqrt{G:Z(G)}$ \citep{MurTPPData2026}, and represents only a small selection or sample from an original dataset for groups of order $1$ to $1000$ (with a few exceptions in this range) containing $49024$ entries \citep{MurTPPData2026}. To minimise duplication Table 1 summarises the relevant data for a selection of groups of nilpotency class $2$ for which $\rho_0$ is known that are \emph{not} $p$-groups with cyclic commutator subgroups of order $p$ (Appendix A Tables \ref{tbl:TblRho0GroupsOfTypeIIa} and \ref{tbl:TblRho0GroupsOfTypeIIb}), with the group-theoretic information obtained from the GAP computer algebra system \citep{GAP}.

\begin{theorem}\label{ThmNilpClassTwo} If $G$ is a group of nilpotency class $2$ then
\begin{equation}\label{SubTPPNilpClassTwo}
\rho_0(G) < \sqrt{G: Z(G)}.
\end{equation}
\end{theorem}
\begin{proof}\textup{(1)} The proof will be by induction on $|G|$ where $G$ is a group of nilpotency class $2$.

\vspace{1mm}

(Base case) The smallest groups of nilpotency class $2$ are the extraspecial groups of order $8 = 2^{1 + 2}$, namely, the dihedral group $D_8$ (GAP ID [8, 3]) and the quaternion group $Q_8$ (GAP ID [8, 4]): these contain cyclic subgroups of index $2$ and, by \citep[Corollary 4.3]{Mur2025}, satisfy $\rho_0(G) = 1 < \sqrt{G:Z(G)} = \sqrt{2^2} = 2$.  Note that by Proposition \ref{PropRho0p6} $\rho_0 = 1$ is also true for all extraspecial groups of order $p^3$ for all primes $p$.

\vspace{1mm}

(Inductive hypothesis) Suppose (\ref{SubTPPNilpClassTwo}) is true for all groups of nilpotency class $2$ of order $ < |G|$, where $G$ is of nilpotency class $2$ and $|G| > 8$.

\vspace{1mm}

Let $(S, T, U)$ be a maximal non-trivial subgroup TPP triple of $G$, that is, $\frac{|S||T||U|}{|G|} = \rho_0(G) > 1$. Choosing any member, $S$, say, let $S_0 = S \cap Z(G)$, where $S_0$ is normal in $G$ (it is a subgroup of central elements of $G$). There are two cases: \textup{(i)} $S_0 = \{1\}$ and \textup{(ii)} $S_0 > \{1\}$, where in both cases $S_0$ is proper in $S$ and $Z(G)$, i.e. $S_0 < S$ and $S < Z(G)$.

\vspace{1mm}

Consider \textup{(i)}, $S_0 = \{1\}$. By Proposition \ref{PropTPPPermSubsets} and Observation \ref{ObsNilpClassTwo} $S$ is an abelian non-normal subgroup of $G$, and the product $SZ(G)$ is an abelian normal subgroup of $G$. Let $H = SZ(G)T$, where $H$ is normal in $G$. There are two subcases: \textup{(i.a)} $H < G$ and \textup{(i.b)} $H = G$.

\vspace{1mm}

Consider \textup{(i.a)}, $H < G$. Letting $U_0 = U \cap H$ and $U_1 = U \cap G\backslash H$ and applying Proposition \ref{PropTPPSplit} $\rho_0(G) = \frac{|S||T||U|}{|G|} \leq \rho_0(H)$. By Proposition \ref{PropTPPPermSubsets} $S$ and $T$ don't centralise each other, and there is a normal series for $H$ given by $\{1\} < H' = [SZ(G)T, SZ(G)T] = [S, T] \leq G' \leq Z(G) \leq Z(H) < H$, which yields a central series $\{1\} < H' \leq Z(H) < H$ of length $2$ and shows that $H$ is of nilpotency class $2$. The inductive hypothesis yields $\rho_0(H) < \sqrt{H:Z(H)}$, and the conditions $|H| < |G|$ and $|Z(H)| \geq |Z(G)|$ imply $\rho_0(G) \leq \rho_0(H) < \sqrt{H:Z(H)} \leq \sqrt{G:Z(G)}$.

\vspace{1mm}

Now consider \textup{(i.b)} $H = SZ(G)T = STZ(G) = \langle S, T, Z(G) \rangle = G$. Let $N = SZ(G) \cap T$. Then $N$ is normal in $G$ as it is normal in $SZ(G)$ and $T$, which generate $G$, and also $N \subseteq T$. Supposing $N > \{1\}$, by Lemma \ref{LemNeuQuoSubTPP}, the quotient $G/N$, which is a group of order $< |G|$, realises the subgroup TPP triple $(T/N,SN/N, UN/N)$ of type $(|T|/|N|,|S|,|U|)$ and $\rho_0(G/N) \geq \rho_0(G)$, and by the inductive hypothesis and Lemma \ref{LemQuoNilpClassTwo} $\rho_0(G) \leq \rho_0(G/N) < \sqrt{G/N:Z(G/N)} \leq \sqrt{G:Z(G)}$. Suppose $N = \{1\}$. Then $SZ(G)T = G$ is an (internal) semidirect product $SZ(G) \rtimes T$, and $G/SZ(G) \cong T$ and $|S||T||Z(G)| = |G|$. Permuting (specifically, transposing) $T$ and $U$, which corresponds to replacing the TPP triple $(S, T, U)$ with the permuted TPP triple $(S, U, T)$, it can be shown that $G/SZ(G) \cong U$, which implies that $T \cong U$ and $|T| = |U|$. Using the same approach with the permuted triples $(T, S, U)$ and $(T, U, S)$ it can be shown that $|S| = |U|$, and by transitivity $|S| = |T| = |U|$. Finally, from $|G| = |S||T||Z(G)| = |S|^2|Z(G)|$ it follows that $|S| = |T| = |U| = \sqrt{\frac{|G|}{|Z(G)|}} = \sqrt{G:Z(G)}$ and $\rho_0(G) = \frac{|S||T||U|}{|G|} = \frac{|S|^3}{|G|} = \frac{\sqrt{G:Z(G)}}{|Z(G)|} < \sqrt{G:Z(G)}$.

\vspace{1mm}

Consider \textup{(ii)} $S_0 = S \cap Z(G) > \{1\}$, where $S_0 < S, Z(G)$. Then $G/S_0$ is a group of order $< |G|$ and nilpotency class $2$ as there is a central series of length $2$ given by $\{1\} < (G/S_0)' = [G/S_0, G/S_0] = G'S_0/S_0 \leq Z(G/S_0) < G/S_0$. The inductive hypothesis and Lemma \ref{LemQuoNilpClassTwo} yield $\rho_0(G/S_0) < \sqrt{G/S_0:Z(G/S_0)} \leq \sqrt{G:Z(G)}$. Lemma \ref{LemNeuQuoSubTPP} implies that $G/S_0$ realises the subgroup TPP triple $(S/S_0, TS_0/S_0, US_0/S_0)$ of type $(|S|/|S_0|,|T|,|U|)$ and $\rho_0(G) = \frac{|S||T||U|/|S_0|}{|G|/|S_0|} \leq \rho_0(G/S_0) < \sqrt{G:Z(G)}$. By permutation invariance the same conclusion follows from supposing $T_0 = T \cap Z(G) > \{1\}$ or $U_0 = U \cap Z(G) > \{1\}$.

\end{proof}

The result may be compared with a basic fact from the character theory of finite groups that if $d$ is the degree of an (irreducible) character of a group $G$ then $d \leq \sqrt{G:Z(G)}$. For the case of groups of nilpotency class $2$ the bound is a generalisation of the most general known upper bound for $\rho_0 \leq \rho$ implied by \citep{Neu2011}[Corollary 3.2] for any class of groups:

\begin{align}\rho_0(G) \leq \rho(G) &\leq \frac{\left( \frac{1 + \sqrt{1 + 8|G|}}{4}  \right)^3}{|G|} \sim \frac{3}{8} + \frac{1}{2\sqrt{2}}\sqrt{|G|}.
\end{align}

\section{$p$-Groups with a Cyclic Commutator Subgroup of Order $p$}\label{SecGroupsOfTypeII}

The $p$-groups with a cyclic commutator subgroup of order $p$ are, as proved in Lemma \ref{LemCycNormSubOrderp}, of nilpotency class $2$. Any member $G$ of this class of order, $p^n$ for $n \geq 3$, say, contains a centre of order at least $p$, and by Theorem \ref{ThmNilpClassTwo}:

\begin{align}\rho_0(G) < p^{\frac{n - 1}{2}}.
\end{align}

A better and sharper bound of $p$ is suggested by the available data for $p$-groups of this class of order up to $1000$ as presented in Appendix A Tables \ref{tbl:TblRho0GroupsOfTypeIIa} and \ref{tbl:TblRho0GroupsOfTypeIIb} \citep[subset of Tables 1 \& 2]{HM}, specifically, $2$- and $3$-groups. It will be noted that in this sample generally $\rho_0 = 1$, except for an extraspecial $2$-group of order $2^{1 + 4} = 32$ (GAP ID \text{[32, 49]}) that reaches the upper bound of $\rho_0 = 2$, and an extraspecial $3$-group of order $3^{1 + 4} = 243$ (GAP ID \text{[243, 65]}) that reaches the upper bound of $\rho_0 = 3$.

\begin{theorem}\label{ThmRho0pGroupCycComm} If $G$ is a $p$-group with a cyclic commutator subgroup $G'$ of order $p$ then
\begin{equation}\label{EqnRho0pGroupCycComm} \rho_0(G) \leq p.
\end{equation}
\end{theorem}
\begin{proof} The proof will be by induction on $|G| = p^n$, $n \geq 3$, where $G$ is a $p$-group with a cyclic commutator subgroup $G'$ of order $p$, and $p$ is an arbitrary fixed prime. Recall that, as shown in Lemma \ref{LemCycNormSubOrderp}, they satisfy \ref{NilpClassTwo}. The base case is extraspecial groups of order $8$, which have already been treated in Theorem \ref{ThmNilpClassTwo} and shown to have $\rho_0 = 1 < p$. Note that by Proposition \ref{PropRho0p6} $\rho_0 = 1$ is also true for all extraspecial groups of order $p^3$ for all primes $p$.

\vspace{1mm}

(Inductive hypothesis) Suppose \ref{EqnRho0pGroupCycComm} is true for all $p$-groups of order $< p^n$ with a cyclic commutator subgroup of order $p$, for $n \geq 5$. 

\vspace{1mm}

The approach taken here will be analogous to that used in Theorem \ref{ThmNilpClassTwo}. \textup{(*)} Let $(S, T, U)$ be a maximal non-trivial subgroup TPP triple of $G$ such that $\frac{|S||T||U|}{|G|} = \rho_0(G) > 1$, that is, $|S||T||U| \geq p|G| = p^{n + 1}$. Note two facts here: first, by Proposition \ref{PropTPPPermSubsets} no two of $S, T, U$ normalise, or, in particular, centralise, each other, so that $[S, T] = [S, U] = [T, U] = G'$ as $G'$ is cyclic of order $p$; and, second, by Observation \ref{ObsNilpClassTwo} $S, T, U$ are all elementary abelian.

\vspace{1mm}

Let $S_0 = S \cap Z(G)$. Then $S_0$ is normal in $G$ as it is central in $G$. There are two cases to consider: \textup{(i)} $S_0 = \{1\}$ and \textup{(ii)} $S_0 > \{1\}$, where in both cases $S_0 < S, Z(G)$. The easiest case is \textup{(ii)} $S_0 > \{1\}$. Here, $|S_0| \geq p$ and there is a quotient $G/S_0$ of order $< p^n$ realising the subgroup TPP triple $(S/S_0, TS_0/S_0, US_0/S_0)$ and with a commutator subgroup $(G/S_0)'$ given by $(G/S_0)' = [G/S_0, G/S_0] = G'S_0/S_0 \cong G'/(G' \cap S_0)$. If $G' \subseteq S_0$ then $G/S_0$ is abelian and $(G/S_0)'$ is trivial, and by Corollary \ref{Corr2LemNeuQuoSubTPP} $\rho_0(G) = 1$. If $G' \not\subseteq S_0$ then $(G/S_0)' \cong C_p$ and by Lemma \ref{LemNeuQuoSubTPP} and the inductive hypothesis, satisfies $\rho_0(G) \leq \rho_0(G/S_0) \leq p$.

\vspace{1mm}

Consider case \textup{(i)}, $S_0 = \{1\}$. As in Theorem \ref{ThmNilpClassTwo} $T \cap Z(G) = \{1\}$ and $U \cap Z(G) = \{1\}$ may also be supposed. Let $H = SZ(G)T = \langle S, T, Z(G) \rangle$. Note that by Observation \ref{ObsNilpClassTwo} $S, T, U$ are all elementary abelian and $H$ is normal in $G$ . There are two subcases \textup{(i.a)} $H < G$ and \textup{(i.b)} $H = G$. Consider \textup{(i.a)}, $H < G$. Then $H$ is of order $< p^n$, and of nilpotency class $2$ as there is a central series of length $2$ given by $\{1\} < H' = [H, H] \leq Z(H) < H$, where $H' = [SZ(G)T, SZ(G)T] = [S, T] = G' \cong C_p$, and by the inductive hypothesis and Proposition \ref{PropTPPSplit} $\rho_0(G) \leq \rho_0(H) \leq p$.

\vspace{1mm}

Now consider \textup{(i.b)} $H = SZ(G)T = \langle S, T, Z(G) \rangle = G$. It will be shown, by constructing in parallel another normal subgroup $K = SG'T = STG'$ of $G$, and comparing it with $H$ while using the assumption that $G'$ is cyclic of order $p$, that, first, $G$ must be extraspecial, and, second, that $\rho_0(G) \leq p$. First, the possibility that $K < G$ can be dealt with using the same argument above for the subcase \textup{(i.a)}, with $K$ replacing $H$ and $G'$ replacing $Z(G)$, which implies that $\rho_0(G) \leq \rho_0(K) \leq p$. So consider $K = SG'T = G = SZ(G)T$, where $SG'$ is normal in $G$. As shown in Theorem \ref{ThmNilpClassTwo} any non-trivial intersection $N = SZ(G) \cap T > \{1\}$ can be used, together with the inductive hypothesis, to construct a quotient $G/N$ that has either a trivial or cyclic-of-order-$p$ commutator subgroup, so that $\rho_0(G) \leq \rho_0(G/N) \leq p$. So let $SZ(G) \cap T = \{1\}$. Then $G$ is a semidirect product $G = SZ(G) \rtimes T$. Now, as $S \cap G' = \{1\}$ and $SG' \cap T = \{1\}$, the condition $K = SG'T = SZ(G)T$ implies that $G$ is also the semidirect product $G = SG' \rtimes T$, which implies that $SG' \cong SZ(G)$ and $G' \cong Z(G) \cong C_p$. The Frattini subgroup $\Phi(G)$, which is the intersection of all the maximal subgroups of $G$ and consists of all the non-generating elements of $G$, is the product $\Phi(G) = G^pG'$, where $G^p$ is the subgroup $\langle g^p \mid g \in G \rangle$ generated by all the $p$-th powers of elements of $G$, and $G/\Phi(G)$ is elementary abelian. In the case being considered $G' \cong C_p$, which means that all commutators $[x, y] \in G'$ are of order $p$, and \ref{EqCommId2} yields $[g^p, x] = [g, x]^p = 1$ for any $g^p \in G^p$ and any $x \in G$. So $G^p \leq Z(G) = G'$, implying that $G' = Z(G) = \Phi(G) \cong C_p$ and $G$ is extraspecial of order $p^{1 + 2m}$ for some $m \geq 2$ such that $n = 2m + 1$. By permuting $S, T, U$, there are also semidirect products $G = SZ(G) \rtimes U = TZ(G) \rtimes S =  TZ(G) \rtimes U$, which imply that $S \cong T \cong U$ and $|S| = |T| = |U| = \sqrt{G:Z(G)} = p^m$. However, by Lemma \ref{ExspGroupsMainRho0} this is only possible for $1 \leq m \leq 2$, implying by Proposition \ref{PropRho0p6} that $\rho_0(G) \leq p$.

\end{proof}

This leads to an immediate corollary for extraspecial groups.

\begin{corollary}\label{CorrRho0pGroupCycComm} If $G$ is an extraspecial group then $\rho_0(G) \leq p$.
\end{corollary}

\section{$p$-Groups Of Nilpotency Class $2$ with a Large Centre}\label{SecGroupsOfTypeIII}

The groups considered here are $p$-group of nilpotency class $2$ that satisfy $p^2 \leq G:Z(G) \leq p^3$ such that $G/Z(G)$ is isomorphic to either $C_p \times C_p$ or $C_{p^2} \times C_p$. The available subgroup TPP ratio data for these groups can be found as a subset of Table \ref{tbl:TblRho0GroupsOfTypeI}, and Tables \ref{tbl:TblRho0GroupsOfTypeIIa} and \ref{tbl:TblRho0GroupsOfTypeIIb}. Although Theorem \ref{ThmNilpClassTwo} implies that $\rho_0 \leq p$ here, a stronger result of $\rho_0 = 1$ follows from Proposition \ref{PropAbCentQuop2p3}.

\begin{theorem}\label{ThmRho0pGroupLargeCentre} If $G$ is a $p$-group of nilpotency class $2$ such that either $G/Z(G) \cong C_p \times C_p$ or $G/Z(G) \cong C_{p^2} \times C_p$ then $\rho_0(G) = 1$.
\end{theorem}
\begin{proof} Let $G$ be as given. Theorem \ref{ThmNilpClassTwo} implies that $\rho_0(G) \leq p$, but the stronger result $\rho_0(G) = 1$ follows from Proposition \ref{PropAbCentQuop2p3} as $G$ contains abelian normal (and maximal) subgroups of index $p$ and \citep[Corollary 4.3]{Mur2025}.
\end{proof}

\section{$p$-Groups of Nilpotency Class $2$ with Small Character Degrees}\label{SecGroupsOfTypeIV}

As proved in \citep[Theorem 4.1, Corollary 4.3]{Mur2025} and used in Theorem \ref{ThmRho0pGroupLargeCentre} groups with large abelian normal subgroups of prime index $p$ tend to have a correspondingly small subgroup TPP ratio $\rho_0$ not exceeding $p$. There is an interesting fact from the character theory of finite groups, that the degree of an irreducible character of a group with an abelian (though, not necessarily normal) subgroup of index $n$, where $n$ is not necessarily prime, cannot exceed $n$. The "smallness" of character degrees of a group also seems to play a role in bounding the value of $\rho_0$, at least indirectly by reflecting the existence of large abelian subgroups. For $p$-groups of nilpotency class $2$ with $\text{cd}(G) = \{1, p\}$ and $G/Z(G) \not \cong C_p \times C_p \times C_p$, for which a selection of relevant data is presented in Tables \ref{tbl:TblRho0GroupsOfTypeI}, \ref{tbl:TblRho0GroupsOfTypeIIa} and \ref{tbl:TblRho0GroupsOfTypeIIb}, there is the following result.

\begin{theorem} If $G$ is a $p$-group of nilpotency class $2$ with character degree set $\text{cd}(G) = \{1, p\}$ such that $G/Z(G) \not \cong C_p \times C_p \times C_p$ then $\rho_0(G) = 1$.
\end{theorem}	
\begin{proof}Let $G$ be as given. By \citep[Theorem C 4.8]{IP} any (nonabelian) $p$-group with $\text{cd}(G) = \{1, p\}$ either contains an abelian normal (and maximal) subgroup of index $p$, or has a centre of index $p^3$. In the first case $\rho_0 = 1$ by \citep[Corollary 4.3]{Mur2025}. In the second case $G/Z(G)$ must be isomorphic to $C_p \times C_p$ or $C_{p^2} \times C_p$, and by Theorem \ref{ThmRho0pGroupLargeCentre} $\rho_0(G) = 1$.
\end{proof}

\hypertarget{bib}{}

\newpage
\begin{appendices}
\section{Data Tables}
\label{appendix}
\begin{table}[H]
  \begin{center}
    \caption{Known subgroup TPP ratio data ($\rho_0$) for a selection of groups of nilpotency class $2$ which do not have a cyclic commutator subgroup of order $p$}
    \label{tbl:TblRho0GroupsOfTypeI}
    \begin{tabular}{lllllll}
      $\bm{|G|}$ & \multicolumn{1}{l}{\textbf{GAP ID}} & \multicolumn{1}{l}{\textbf{Structure Description}} & \multicolumn{1}{l}{$\bm{|Z(G)|}$} & \multicolumn{1}{l}{$\bm{\sqrt{G:Z(G)}}$} & \multicolumn{1}{l}{$\bm{\textbf{cd}(G)}$} & \multicolumn{1}{l}{$\bm{\rho_0}$} \\
      \toprule
        \multirow{2}{*}{$24$} &
            [24, 10] & $C_3 \times D_8$ & $6$ & $2$ & $\{1, 2\}$ & $1$ \\
          & [24, 11] & $C_3 \times Q_8$ & $6$ & $2$ & $\{1, 2\}$ & $1$ \\
        \hline
        \multirow{9}{*}{$32$} & 
            [32, 28] & $(C_4 \times C_2 \times C_2) \rtimes C_2$ & $4$ & $2\sqrt{2}$ & $\{1, 2\}$ & $1$ \\
        & [32, 29] & $(C_2 \times Q_8) \rtimes C_2$ & $4$ & $2\sqrt{2}$ & $\{1, 2\}$ & $1$ \\
        & [32, 30] & $(C_4 \times C_2 \times C_2) \rtimes C_2$ & $4$ & $2\sqrt{2}$ & $\{1, 2\}$  & $1$ \\
        & [32, 31] & $(C_4 \times C_4) \rtimes C_2$ & $4$ & $2\sqrt{2}$ & $\{1, 2\}$ & $1$ \\
        & [32, 32] & $(C_2 \times C_2) . (C_2 \times C_2 \times C_2)$ & $4$ & $2\sqrt{2}$ & $\{1, 2\}$ & $1$ \\
        & [32, 33] & $(C_4 \times C_4) \rtimes C_2$ & $4$ & $2\sqrt{2}$ & $\{1, 2\}$ & $1$ \\
        & [32, 34] & $(C_4 \times C_4) \rtimes C_2$ & $4$ & $2\sqrt{2}$ & $\{1, 2\}$ & $1$ \\
        & [32, 35] & $C_4 \rtimes Q_8$ & $4$ & $2\sqrt{2}$ & $\{1, 2\}$ & $1$ \\
        \hline
        \multirow{1}{*}{$64$} & 
            [64, 226] & $D_8 \rtimes D_8$ & $4$ & $4$ & $\{1, 2, 4\}$ & $2$ \\
        \hline
        \multirow{5}{*}{$128$} & 
            [128, 1135] & $(C_2 \times C_2 \times C_2 \times D_8) \rtimes C_2$ & $8$ & $4$ & $\{1, 2, 4\}$ & $2$ \\
         & [128, 1142] & $(C_2 \times C_2 \times ((C_4 \times C_2) \rtimes C_2)$ & $8$ & $4$ & $\{1, 2, 4\}$ & $2$ \\
         & [128, 1165] & $(C_2 \times C_2 \times C_2 \times D_8) \rtimes C_2$ & $8$ & $4$ & $\{1, 2, 4\}$ & $2$ \\
         & [128, 2194] & $C_2 \times D_8 \times D_8$ & $8$ & $4$ & $\{1, 2, 4\}$ & $2$ \\
         & [128, 2213] & $((C_4 \times D_8) \rtimes C_2) \rtimes C_2$ & $8$ & $4$ & $\{1, 2, 4\}$ & $2$ \\
    \end{tabular}
  \end{center}
\end{table}

\newpage

\begin{table}[H]
  \begin{center}
    \caption{Known subgroup TPP ratio data ($\rho_0$) for a selection of small $2$-groups with a cyclic commutator subgroup of order $2$}
    \label{tbl:TblRho0GroupsOfTypeIIa}
    \begin{tabular}{llllllll}
      $\bm{|G|}$ & \multicolumn{1}{l}{\textbf{GAP ID}} & \multicolumn{1}{l}{\textbf{Structure Description}} & \multicolumn{1}{l}{$\bm{|Z(G)|}$} & \multicolumn{1}{l}{$\bm{\sqrt{G:Z(G)}}$} & \multicolumn{1}{l}{$\bm{\textbf{cd}(G)}$} & \multicolumn{1}{l}{$\bm{\rho_0}$} \\
      \toprule
        \multirow{2}{*}{$8$} &
            [8, 3] & $D_8$ & $2$ & $2$ & $\{1, 2\}$ & $1$ \\
          & [8, 4] & $Q_8$ & $2$ & $2$ & $\{1, 2\}$ & $1$ \\
        \hline
        \multirow{6}{*}{$16$} & 
            [16, 3] & $(C_4 \times C_2) \rtimes C_2$ & $4$ & $2$ & $\{1, 2\}$ & $1$ \\
          & [16, 4] & $C_4 \times C_4$ & $4$ & $2$ & $\{1, 2\}$ & $1$ \\
          & [16, 6] & $C_8 \rtimes C_2$ & $4$ & $2$ & $\{1, 2\}$ & $1$ \\
          & [16, 11] & $C_2 \times D_8$ & $4$ & $2$ & $\{1, 2\}$ & $1$ \\
          & [16, 12] & $C_2 \times Q_8$ & $4$ & $2$ & $\{1, 2\}$ & $1$ \\
          & [16, 13] & $(C_4 \times C_2) \rtimes C_2$ & $4$ & $2$ & $\{1, 2\}$ & $1$ \\
        \hline
        \multirow{5}{*}{$32$} &
            [32, 2] & $(C_4 \times C_2) \rtimes C_4$ & $8$ & $2$ & $\{1, 2\}$ & $1$ \\
          & [32, 4] & $C_8 \rtimes C_4$ & $8$ & $2$ & $\{1, 2\}$ & $1$ \\
          & [32, 5] & $(C_8 \times C_2) \rtimes C_2$ & $8$ & $2$ & $\{1, 2\}$ & $1$ \\
          & [32, 12] & $C_4 \times C_8$ & $8$ & $2$ & $\{1, 2\}$ & $1$ \\
          & [32, 17] & $C_{16} \rtimes C_2$ & $8$ & $2$ & $\{1, 2\}$ & $1$ \\
          & [32, 22] & $C_2 \times ((C_4 \times C_2) \rtimes C_2)$ & $8$ & $2$ & $\{1, 2\}$ & $1$ \\
          & [32, 23] & $C_2 \times (C_4 \rtimes C_4)$ & $8$ & $2$ & $\{1, 2\}$ & $1$ \\
          & [32, 24] & $(C_4 \times C_4) \rtimes C_2$ & $8$ & $2$ & $\{1, 2\}$ & $1$ \\
          & [32, 25] & $C_4 \times D_8$ & $8$ & $2$ & $\{1, 2\}$ & $1$ \\
          & [32, 26] & $C_4 \times Q_8$ & $8$ & $2$ & $\{1, 2\}$ & $1$ \\
          & [32, 37] & $C_2 \times (C_8 \rtimes C_2)$ & $8$ & $2$ & $\{1, 2\}$ & $1$ \\
          & [32, 38] & $(C_8 \times C_2) \rtimes C_2$ & $8$ & $2$ & $\{1, 2\}$ & $1$ \\
          & [32, 46] & $C_2 \times C_2 \times D_8$ & $8$ & $2$ & $\{1, 2\}$ & $1$ \\
          & [32, 47] & $C_2 \times C_2 \times Q_8$ & $8$ & $2$ & $\{1, 2\}$ & $1$ \\
          & [32, 48] & $C_2 \times ((C_4 \times C_2) \times C_2)$ & $8$ & $2$ & $\{1, 2\}$ & $1$ \\
          & [32, 49] & $(C_2 \times C_2 \times C_2) \rtimes (C_2 \times C_2)$ & $2$ & $4$ & $\{1, 4\}$ & $2$ \\
          & [32, 50] & $(C_2 \times Q_8) \rtimes C_2$ & $2$ & $4$ & $\{1, 2\}$ & $1$ \\
    \end{tabular}
  \end{center}
\end{table}

\begin{table}[H]
  \begin{center}
    \caption{Known subgroup TPP ratio data ($\rho_0$) for a selection of small $3$-groups with a cyclic commutator subgroup of order $3$}
    \label{tbl:TblRho0GroupsOfTypeIIb}
    \begin{tabular}{llllllll}
      $\bm{|G|}$ & \multicolumn{1}{l}{\textbf{GAP ID}} & \multicolumn{1}{l}{\textbf{Structure Description}} & \multicolumn{1}{l}{$\bm{|Z(G)|}$} & \multicolumn{1}{l}{$\bm{\sqrt{G:Z(G)}}$} & \multicolumn{1}{l}{$\bm{\textbf{cd}(G)}$} & \multicolumn{1}{l}{$\bm{\rho_0}$} \\
      \toprule
        \multirow{2}{*}{$27$} &
            [27, 3] & $(C_3 \times C_3) \rtimes C_3$ & $3$ & $3$ & $\{1, 3\}$ & $1$ \\
          & [27, 4] & $C_9 \rtimes C_3$ & $3$ & $3$ &  $\{1, 3\}$ &$1$ \\
        \hline
        \multirow{1}{*}{$243$} &
            [243, 65] & $(C_3 \rtimes ((C_3 \times C_3) \rtimes C_3)) \times C_3$ & $3$ & $9$ & $\{1, 9\}$ & $3$ \\
    \end{tabular}
  \end{center}
\end{table}

\end{appendices}


\begin{thebibliography}{}
\bibitem{AGW}J. D. Adler, and M. Garlow and E. R. Wheland, 'Groups of order $p^4$ made less difficult', arXiv:1611.00461, 2016. \url{https://doi.org/10.48550/arXiv.1611.00461}
\bibitem{AB}J. L. Alperin and R. B. Bell, \emph{Groups and Representations}, Springer, 1991.
\bibitem{Ber}Y. Berkovich, \emph{Groups of Prime Power Order}, Volume 1, De Gruyter, 2008.
\bibitem{CU}H. Cohn and C. Umans, 'A group-theoretic approach to fast matrix multiplication', in: \emph{Proceedings of the 44th Annual Symposium on Foundations of Computer Science} (Cambridge, Massachusetts 2003), IEEE Computer Society (2003), 438-449. \url{https://dl.acm.org/doi/10.5555/946243.946301}
\bibitem{Cra}, D. A. Craven, The Theory of $p$-Groups, Lecture notes, 2008. \url{https://web.mat.bham.ac.uk/D.A.Craven/docs/lectures/pgroups.pdf}
\bibitem{Hea}D. R. Heath-Brown, 'A Note on the Chevalley-Warning Theorems', arXiv:1009.3764, 2010. \url{https://doi.org/10.48550/arXiv.1009.3764}
\bibitem{IP}I. M. Isaacs and D. S. Passman, 'A characterization of groups in terms of the degrees of their characters', \emph{Pacific Journal of Mathematics}, 15(3): 877-903, 1965. \url{http://dx.doi.org/10.2140/pjm.1965.15.877}
\bibitem{GAP}The GAP Group, GAP - Groups, Algorithms, and Programming, Version 4.15.1, 2025. \url{https://www.gap-system.org}
\bibitem{HM}I. Hedtke and S. Murthy, 'Search and test algorithms for triple product property triples', \emph{Groups - Complexity - Cryptology}, 4(1): 111-133, 2012. \url{https://doi.org/10.1515/gcc-2012-0006}
\bibitem{Mur2025}S. R. Murthy, 'A note on the triple product property for finite groups with abelian normal subgroups of prime index', arXiv:2512.16730, 2025. \url{https://doi.org/10.48550/arXiv.2512.16730}
\bibitem{MurTPPData2026}S. R. Murthy, TPP triples search data repository, 2026. \url{https://github.com/sr-murthy/TPP}
\bibitem{Neu2011}P. M. Neumann, 'A note on the triple product property for subsets of finite groups', \emph{LMS J. Computation and Mathematics}, 14:232-237, 2011. \url{https://doi.org/10.1112/S1461157010000288}
\bibitem{Neu2012}P. M. Neumann, Private communication, 2012.
\bibitem{Wil}M. W. Wild, 'The groups of order sixteen made easy', \emph{The American Mathematical Monthly}, 112(1): 20-31, 2005. \url{https://doi.org/10.2307/30037381}
\end{thebibliography}
\end{document}